\documentclass[11pt,english]{article}
\usepackage[T1]{fontenc}
\usepackage[latin9]{inputenc}
\usepackage{geometry}
\usepackage{float}
\usepackage{units}
\usepackage{amsthm}
\usepackage{amsmath}
\usepackage{amssymb}
\usepackage{graphicx}
\usepackage[all]{xy}

\makeatletter
\numberwithin{equation}{section}
\theoremstyle{plain}
\newtheorem{thm}{\protect\theoremname}[section]
  \theoremstyle{definition}
  \newtheorem{defn}[thm]{\protect\definitionname}
  \theoremstyle{plain}
  \newtheorem{cor}[thm]{\protect\corollaryname}
  \theoremstyle{remark}
  \newtheorem*{acknowledgement*}{\protect\acknowledgementname}
  \theoremstyle{definition}
  \newtheorem*{example*}{\protect\examplename}
  \theoremstyle{plain}
  \newtheorem{prop}[thm]{\protect\propositionname}
  \theoremstyle{remark}
  \newtheorem*{rem*}{\protect\remarkname}
  \theoremstyle{remark}
  \newtheorem{claim}[thm]{\protect\claimname}
  \theoremstyle{plain}
  \newtheorem{lem}[thm]{\protect\lemmaname}
  \theoremstyle{plain}
  \newtheorem*{thm*}{\protect\theoremname}
  \theoremstyle{remark}
  \newtheorem{rem}[thm]{\protect\remarkname}
\newenvironment{lyxlist}[1]
{\begin{list}{}
{\settowidth{\labelwidth}{#1}
 \setlength{\leftmargin}{\labelwidth}
 \addtolength{\leftmargin}{\labelsep}
 }}
{\end{list}}

\usepackage{txfonts}
\usepackage{ifsym}
\usepackage{graphicx}
\usepackage{url}
\usepackage[perpage,symbol]{footmisc}

\DefineFNsymbols*{lamportnostar}[math]{\dagger\ddagger\S\P\|{\dagger\dagger}{\ddagger\ddagger}}
\setfnsymbol{lamportnostar}

\DeclareMathOperator{\im}{im}
\DeclareMathOperator{\sgn}{sgn}
\DeclareMathOperator{\Spec}{Spec}
\DeclareMathOperator{\lk}{lk}
\DeclareMathOperator{\overlap}{overlap}
\DeclareMathOperator{\nul}{null}
\DeclareMathOperator{\rank}{rank}

\DeclareMathOperator{\tr}{trace}
\theoremstyle{remark}
\newtheorem*{rems*}{Remarks}

\theoremstyle{plain}
\newtheorem{claim}[thm]{\protect\claimname}

\theoremstyle{plain}

\renewcommand{\qed}{\hfill \mbox{\raggedright \rule{0.1in}{0.1in}}}

\makeatother

\usepackage{babel}
  \providecommand{\acknowledgementname}{Acknowledgement}
  \providecommand{\claimname}{Claim}
  \providecommand{\corollaryname}{Corollary}
  \providecommand{\definitionname}{Definition}
  \providecommand{\examplename}{Example}
  \providecommand{\lemmaname}{Lemma}
  \providecommand{\propositionname}{Proposition}
  \providecommand{\remarkname}{Remark}
  \providecommand{\theoremname}{Theorem}
\providecommand{\theoremname}{Theorem}

\begin{document}

\title{Isoperimetric Inequalities in Simplicial Complexes}

\author{Ori Parzanchevski, Ron Rosenthal and Ran J.\ Tessler}
\maketitle
\begin{abstract}
In graph theory there are intimate connections between the expansion
properties of a graph and the spectrum of its Laplacian. In this paper
we define a notion of combinatorial expansion for simplicial complexes
of general dimension, and prove that similar connections exist between
the combinatorial expansion of a complex, and the spectrum of the
high dimensional Laplacian defined by Eckmann. In particular, we present
a Cheeger-type inequality, and a high-dimensional Expander Mixing
Lemma. As a corollary, using the work of Pach, we obtain a connection
between spectral properties of complexes and Gromov's notion of geometric
overlap. Using the work of Gunder and Wagner, we give an estimate
for the combinatorial expansion and geometric overlap of random Linial-Meshulam
complexes.
\end{abstract}

\section{Introduction}

It is a cornerstone of graph theory that the expansion properties
of a graph are intimately linked to the spectrum of its Laplacian.
In particular, the discrete Cheeger inequalities \cite{Tan84,Dod84,AM85,Alo86}
relate the spectral gap of a graph to its Cheeger constant, and the
Expander Mixing Lemma \cite{friedman1987expanding,AC88,beigel1993fault}
relates the extremal values of the spectrum to discrepancy in the
graph (see \eqref{eq:mixing-graphs}) and to its mixing properties.

In this paper we define a notion of expansion for simplicial complexes,
which generalizes the Cheeger constant and the discrepancy in graphs.
We then study its relations to the spectrum of the high dimensional
Laplacian defined by Eckmann \cite{Eck44}, and present a high dimensional
Cheeger inequality and a high dimensional Expander Mixing Lemma.

\bigskip{}

This study is closely related to the notion of \emph{high dimensional
expanders}. A family of graphs $\left\{ G_{i}\right\} $ with uniformly
bounded degrees is said to be a family of \emph{expanders} if their
Cheeger constants $h\left(G_{i}\right)$ are uniformly bounded away
from zero. By the discrete Cheeger inequalities \eqref{eq:graph_cheeger_ineq},
this is equivalent to having their spectral gaps $\lambda\left(G_{i}\right)$
uniformly bounded away from zero. Thus, combinatorial expanders and
spectral expanders are equivalent notions. We refer to \cite{HLW06,Lub12}
for the general background on expanders and their applications.

It is desirable to have a similar situation in higher dimensions,
but at least as of now, it is not clear what is the ``right'' notion
of ``high dimensional expander''. One generalization of the Cheeger
constant to higher dimensions is the notion of \emph{coboundary expansion},
originating in \cite{LM06,Gro10}, and studied under various names
in \cite{meshulam2009homological,dotterrer2010coboundary,MW11,GW12,mukherjee2012cheeger,NR12}.
While in dimension one it coincides with the Cheeger constant, its
combinatorial meaning is somewhat vague in higher dimensions. Furthermore,
it is shown in \cite{GW12} that there exist, in any dimension greater
than one, complexes with spectral gaps bounded away from zero%
\footnote{\ The spectral gap of a complex is defined in Section \ref{sub:gap-def}.%
} and arbitrarily small coboundary expansion; In \cite{mukherjee2012cheeger}
the other direction is settled: there exist coboundary expanding complexes
with arbitrarily small spectral gaps.

Another notion of expansion is Gromov's \emph{geometric overlap property},
originating in \cite{Gro10} and studied in \cite{FGL+11,MW11}. This
notion was shown in \cite{Gro10,MW11} to be related to coboundary
expansion. However, even in dimension one it is not equivalent to
that of expander graphs.

Our definition of expansion suggests a natural notion of ``combinatorial
expanders'', and we show that spectral expanders with complete skeletons
are combinatorial expanders. A theorem of Pach \cite{Pac98} shows
that this notion of combinatorial expansion is also connected to the
geometric overlap property. As an application of our main theorems
we analyze the Linial-Meshulam model of random complexes, and show
that for suitable parameters they form combinatorial and geometric
expanders.

\subsection{Combinatorial expansion and the spectral gap}

The \emph{Cheeger constant} of a finite graph $G=\left(V,E\right)$
on $n$ vertices is usually taken to be 
\[
\varphi\left(G\right)=\min_{{A\subseteq V\atop 0<\left|A\right|\leq\frac{n}{2}}}\frac{\left|E\left(A,V\backslash A\right)\right|}{\left|A\right|}
\]
where $E\left(A,B\right)$ is the set of edges with one vertex in
$A$ and the other in $B$. In this paper, however, we work with the
following version:
\begin{equation}
h\left(G\right)=\min_{0<\left|A\right|<n}\frac{n\left|E\left(A,V\backslash A\right)\right|}{\left|A\right|\left|V\backslash A\right|}.\label{eq:cheeger-graph}
\end{equation}
Since $\varphi\left(G\right)\leq h\left(G\right)\leq2\varphi\left(G\right)$,
defining expanders by $\varphi$ or by $h$ is equivalent. 

The \emph{spectral gap} of $G$, denoted $\lambda\left(G\right)$,
is the second smallest eigenvalue of the \emph{Laplacian} $\Delta^{+}:\mathbb{R}^{V}\rightarrow\mathbb{R}^{V}$,
which is defined by 
\begin{equation}
\left(\Delta^{+}f\right)\left(v\right)=\deg\left(v\right)f\left(v\right)-\sum_{w\sim v}f\left(w\right).\label{eq:graph-laplacian}
\end{equation}
The discrete Cheeger inequalities \cite{Tan84,Dod84,AM85,Alo86} relate
the Cheeger constant and the spectral gap: 
\begin{equation}
\frac{h^{2}\left(G\right)}{8k}\leq\lambda\left(G\right)\leq h\left(G\right),\label{eq:graph_cheeger_ineq}
\end{equation}
where $k$ is the maximal degree of a vertex in $G$.%
\footnote{\ For $\varphi$ they are given by $\frac{\varphi^{2}\left(G\right)}{2k}\leq\lambda\left(G\right)\leq2\varphi\left(G\right).$%
} In particular, the bound $\lambda\leq h$ shows that spectral expanders
are combinatorial expanders. This proved to be of immense importance
since the spectral gap is approachable by many mathematical tools
(coming from linear algebra, spectral methods, representation theory
and even number theory - see e.g.\ \cite{Lub10,Lub12} and the references
within). In contrast, the Cheeger constant is usually hard to analyze
directly, and even to compute it for a given graph is NP-hard \cite{blum1981complexity,matula1990sparsest}.

\bigskip{}

Moving on to higher dimension, let $X$ be an (abstract) simplicial
complex with vertex set $V$. This means that $X$ is a collection
of subsets of $V$, called \emph{cells} (and also \emph{simplexes},\emph{
faces},\emph{ }or \emph{hyperedges}), which is closed under taking
subsets, i.e., if $\sigma\in X$ and $\tau\subseteq\sigma$, then
$\tau\in X$. The \emph{dimension }of a cell $\sigma$ is $\dim\sigma=\left|\sigma\right|-1$,
and $X^{j}$ denotes the set of cells of dimension $j$. The dimension
of $X$ is the maximal dimension of a cell in it. The \emph{degree}
of a $j$-cell (a cell of dimension $j$) is the number of $\left(j+1\right)$-cells
which contain it. Throughout this paper we denote by $d$ the dimension
of the complex at hand, and by $n$ the number of vertices in it.
We shall occasionally add the assumption that the complex has a \emph{complete
skeleton}, by which we mean that every possible $j$-cell with $j<d$
belongs to $X$. 

We define the following generalization of the Cheeger constant:
\begin{defn}
For a finite $d$-complex $X$ with $n$ vertices $V$, 
\[
h\left(X\right)=\min\limits _{V=\coprod_{i=0}^{d}A_{i}}\frac{n\cdot\left|F\left(A_{0},A_{1},\ldots,A_{d}\right)\right|}{\left|A_{0}\right|\cdot\left|A_{1}\right|\cdot\ldots\cdot\left|A_{d}\right|},
\]
where the minimum is taken over all partitions of $V$ into nonempty
sets $A_{0},\ldots,A_{{d}}$, and $F\left(A_{0},\ldots,A_{{d}}\right)$
denotes the set of ${d}$-dimensional cells with one vertex in each
$A_{i}$.
\end{defn}
For $d=1$, this coincides with the Cheeger constant of a graph \eqref{eq:cheeger-graph}.
To formulate an analogue of the Cheeger inequalities, we need a high-dimensional
analogue of the spectral gap. Such an analogue is provided by the
work of Eckmann on discrete Hodge theory \cite{Eck44}. In order to
give the definition we shall need more terminology, and we defer this
to Section \ref{sub:gap-def}%
\footnote{\ The spectral gap appears in Definition \ref{def:The-spectral-gap},
and is given alternative characterizations in Propositions \ref{prop:gap-alternat}
and \ref{prop:spectral-gap-complete}.%
}. The basic idea, however, is the same as for graphs, namely, the
spectral gap $\lambda\left(X\right)$ is the smallest nontrivial eigenvalue
of a suitable Laplace operator. The following theorem, whose proof
appears in Section 4.1, generalizes the upper Cheeger inequality to
higher dimensions:
\begin{thm}[Cheeger Inequality]
\label{thm:Isoperimetric_inequality}For a finite complex $X$ with
a complete skeleton, $\lambda\left(X\right)\leq h\left(X\right)$.$ $
\end{thm}
\begin{rems*}$ $
\begin{enumerate}
\item If the skeleton of $X$ is not complete, then $h\left(X\right)=0$,
since there exist some $\left\{ v_{0},\ldots,v_{d-1}\right\} \notin X^{d-1}$,
and then $F\left(\left\{ v_{0}\right\} ,\left\{ v_{1}\right\} ,\ldots,\left\{ v_{d-1}\right\} ,V\backslash\left\{ v_{0},\ldots,v_{d-1}\right\} \right)=0$.
This suggests that a different definition of $h$ is called for, and
we propose one in Section \ref{sec:Open-Questions}.
\item For a discussion of a possible lower Cheeger inequality, see Section
\ref{sub:Towards-a-lower}.
\end{enumerate}
\end{rems*} \bigskip{}

In \cite{LM06} Linial and Meshulam introduced the following model
for random simplicial complexes: for a given $p=p\left(n\right)\in\left(0,1\right)$,
$X\left(d,n,p\right)$ is a $d$-dimensional simplicial complex on
$n$ vertices, with a complete skeleton, and with every $d$-cell
being included independently with probability $p$. Using the analysis
of the spectrum of $X\left(d,n,p\right)$ in \cite{GW12}, we show
the following:
\begin{cor}
\label{cor:Random_cor}The Linial-Meshulam complexes satisfy the following:
\begin{enumerate}
\item For large enough $C$, a.a.s.\ $h\left(X\left(d,n,\frac{C\log n}{n}\right)\right)\geq\left(C-O\left(\!\sqrt{C}\right)\right)\log n$.
\item For $C<1$, a.a.s.\ $h\left(X\left(d,n,\frac{C\log n}{n}\right)\right)=0$.
\end{enumerate}
\end{cor}
The proof appears in Section \ref{sub:Expansion-in-random}, as part
of Corollary \ref{cor:Random_unified}.

\subsection{Mixing and discrepancy }

The Cheeger inequalities \eqref{eq:cheeger-graph} bound the expansion
along the partitions of a graph, in terms of its spectral gap. However,
the spectral gap alone does not suffice to determine the expansion
between arbitrary sets of vertices. For example, the bipartite Ramanujan
graphs constructed in \cite{LPS88} are regular graphs with very large
spectral gaps, which are bipartite. This means that they contain disjoint
sets $A,B\subseteq V$ of size $\frac{n}{4}$, with $E\left(A,B\right)=\varnothing$.
It turns out that control of the expansion between any two sets of
vertices is possible by observing not only the smallest nontrivial
eigenvalue of the Laplacian, but also the largest one%
\footnote{\ Graphs having both of them bounded are referred to as ``two-sided
expanders'' in \cite{Tao11}.%
}. In particular, the so-called Expander Mixing Lemma (\cite{friedman1987expanding,AC88,beigel1993fault},
see also \cite{HLW06}) states that for a $k$-regular graph $G=\left(V,E\right)$,
and $A,B\subseteq V$,
\begin{equation}
\left|\left|E\left(A,B\right)\right|-\frac{k\left|A\right|\left|B\right|}{n}\right|\leq\rho\cdot\sqrt{\left|A\right|\left|B\right|},\label{eq:mixing-graphs}
\end{equation}
where $\rho$ is the maximal absolute value of a nontrivial eigenvalue
of $kI-\Delta^{+}$. 

The deviation of $\left|E\left(A,B\right)\right|$ from its expected
value $p\left|A\right|\left|B\right|$, where $p=\frac{k}{n}\approx\nicefrac{\left|E\right|}{{n \choose 2}}$
is the edge density, is called the \emph{discrepancy} of $A$ and
$B$. This is a measure of quasi-randomness in a graph, a notion closely
related to expansion (see e.g.\ \cite{Chu97}). In a similar fashion,
if $k$ is the average degree of a $\left(d-1\right)$-cell in $X$,
we call the deviation 
\[
\left|\left|F\left(A_{0},\ldots,A_{d}\right)\right|-\frac{\left|X^{d}\right|}{{n \choose d+1}}\cdot\left|A_{0}\right|\cdot\ldots\cdot\left|A_{d}\right|\right|\approx\left|\left|F\left(A_{0},\ldots,A_{d}\right)\right|-\frac{k\left|A_{0}\right|\cdot\ldots\cdot\left|A_{d}\right|}{n}\right|
\]
the discrepancy of $A_{0},\ldots,A_{d}$ (the question of using $\frac{\left|X^{d}\right|}{{n \choose d+1}}$
or $\frac{k}{n}$ is addressed in Remark \ref{rem:discrepancy}).
The following theorem generalizes the Expander Mixing Lemma to higher
dimensions:
\begin{thm}[Mixing Lemma]
\label{thm:mixing}If $X$ is a $d$-dimensional complex with a complete
skeleton, then for any disjoint sets of vertices $A_{0},\ldots,A_{d}$
one has 
\[
\left|\left|F\left(A_{0},\ldots,A_{d}\right)\right|-\frac{k\cdot\left|A_{0}\right|\cdot\ldots\cdot\left|A_{d}\right|}{n}\right|\leq\rho\cdot\left(\left|A_{0}\right|\cdot\ldots\cdot\left|A_{d}\right|\right)^{\frac{d}{d+1}},
\]
where $k$ is the average degree of a $\left(d-1\right)$-cell in
$X$, and $\rho$ is the maximal absolute value of a nontrivial eigenvalue
of $kI-\Delta^{+}$.
\end{thm}
Here $\Delta^{+}$ is the Laplacian of $X$, which is defined in Section
\ref{sec:Notations-and-definitions}. The proof, and a formal definition
of $\rho$, appear in Section \ref{sub:The-Mixing-Lemma}.

\bigskip{}

A related measure of expansion in graphs is given by the convergence
rate of the random walk on it. As for the discrepancy, it is not enough
to bound the spectral gap but also the higher end of the Laplace spectrum
in order to understand this expansion. For example, on the bipartite
graphs mentioned earlier the random walk does not converge at all.
In \cite{PR12} we suggest a generalization of the notion of random
walk to general simplicial complexes, and study its connection to
the spectral properties of the complex.

\subsection{Geometric overlap}

If a graph $G=\left(V,E\right)$ has a large Cheeger constant, then
given a mapping $\varphi:V\rightarrow\mathbb{R}$, there exists a
point $x\in\mathbb{R}$ which is covered by many edges in the linear
extension of $\varphi$ to $E$ (namely, $x=\mathrm{median}\left(\left\{ \varphi\left(v\right)\,\middle|\, v\in V\right\} \right)$.
This observation led Gromov to define the \emph{geometric overlap
}of a complex \cite{Gro10}:
\begin{defn}
Let $X$ be a $d$-dimensional simplicial complex. The overlap of
$X$ is defined by
\[
\mbox{overlap}\left(X\right)=\min_{\varphi:V\rightarrow\mathbb{R}^{d}}\,\max_{x\in\mathbb{R}^{d}}\,\frac{\#\left\{ \sigma\in X^{d}\,\middle|\, x\in\mathrm{conv}\left\{ \varphi\left(v\right)\,\middle|\, v\in\sigma\right\} \right\} }{\left|X^{d}\right|}.
\]
In other words, $X$ has $\overlap\geq\varepsilon$ if for every simplicial
mapping of $X$ into $\mathbb{R}^{d}$ (a mapping induced linearly
by the images of the vertices), some point in $\mathbb{R}^{d}$ is
covered by at least an $\varepsilon$-fraction of the $d$-cells of
$X$. 
\end{defn}
A theorem of Pach \cite{Pac98}, together with Theorem \ref{thm:mixing}
yield a connection between the spectrum of the Laplacian and the overlap
property. 
\begin{cor}
\label{cor:Geometric_overlap_cor}Let $X$ be a $d$-complex with
a complete skeleton, and denote the average degree of a $\left(d-1\right)$-cell
in $X$ by $k$. If the nontrivial spectrum of the Laplacian of $X$
is contained in $\left[k-\varepsilon,k+\varepsilon\right]$, then
\[
\overlap\left(X\right)\geq\frac{c_{d}^{d}}{e^{d+1}}\left(c_{d}-\frac{\varepsilon\left(d+1\right)}{k}\right),
\]
where $c_{d}$ is Pach's constant from \cite{Pac98}.
\end{cor}
The proof appears in Section \ref{sub:Gromov's-Geometric-Overlap}.
As an application of this corollary, we show that Linial-Meshulam
complexes have geometric overlap for suitable parameters:
\begin{cor}
\label{cor:random-overlap}There exist $\vartheta>0$ such that for
large enough $C$ a.a.s.\ $\overlap\left(X\left(d,n,\frac{C\cdot\log n}{n}\right)\right)>\vartheta$.
\end{cor}
Again, this is a part of Corollary \ref{cor:Random_unified}, which
is proved in Section \ref{sub:Expansion-in-random}.

\bigskip{}
The structure of the paper is as follows: in Section 2 we present
the basic definitions relating to simplicial complexes and their spectral
theory. Section 3 is devoted to proving basic properties of the high
dimensional Laplacians. In Section 4 we prove the theorems and corollaries
stated in the introduction, and discuss the possibility of a lower
Cheeger inequality. Finally, Section 5 lists some open questions.
\begin{acknowledgement*}
The authors would like to thank Alex Lubotzky for initiating our study
of spectral expansion of complexes. We would also like to express
our gratitude for the suggestions made to us by Noam Berger, Konstantin
Golubev, Gil Kalai, Nati Linial, Doron Puder, Doron Shafrir, Uli Wagner,
and Andrzej \.{Z}uk. We are thankful for the support of the ERC.
\end{acknowledgement*}

\section{\label{sec:Notations-and-definitions}Notations and definitions}

Recall that $X$ denotes a finite $d$-dimensional simplicial complex
with vertex set $V$ of size $n$, and that $X^{j}$ denotes the set
of $j$-cells of $X$, where $-1\leq j\leq d$. In particular, we
have $X^{-1}=\left\{ \varnothing\right\} $. For $j\geq1$, every
$j$-cell $\sigma=\left\{ \sigma_{0},\ldots,\sigma_{j}\right\} $
has two possible orientations, corresponding to the possible orderings
of its vertices, up to an even permutation ($1$-cells and the empty
cell have only one orientation). We denote an oriented cell by square
brackets, and a flip of orientation by an overbar. For example, one
orientation of $\sigma=\left\{ x,y,z\right\} $ is $\left[x,y,z\right]$,
which is the same as $\left[y,z,x\right]$ and $\left[z,x,y\right]$.
The other orientation of $\sigma$ is $\overline{\left[x,y,z\right]}=\left[y,x,z\right]=\left[x,z,y\right]=\left[z,y,x\right]$.
We denote by $X_{\pm}^{j}$ the set of oriented $j$-cells (so that
$\left|X_{\pm}^{j}\right|=2\left|X^{j}\right|$ for $j\geq1$ and
$X_{\pm}^{j}=X^{j}$ for $j=-1,0$). 

We now describe the \emph{discrete Hodge theory} due to Eckmann \cite{Eck44}.
This is a discrete analogue of Hodge theory in Riemannian geometry,
but in contrast, the proofs of the statements are all exercises in
finite-dimensional linear algebra. Furthermore, it applies to any
complex, and not only to manifolds.

The space of \emph{$j$-forms on $X$}, denoted $\Omega^{j}\left(X\right)$,\emph{
}is the vector space of skew-symmetric functions on oriented $j$-cells:
\[
\Omega^{j}=\Omega^{j}\left(X\right)=\left\{ f:X_{\pm}^{j}\rightarrow\mathbb{R}\,\middle|\, f\left(\overline{\sigma}\right)=-f\left(\sigma\right)\;\forall\sigma\in X_{\pm}^{j}\right\} .
\]
In particular, $\Omega^{0}$ is the space of functions on $V$, and
$\Omega^{-1}=\mathbb{R}^{\left\{ \varnothing\right\} }$ can be identified
in a natural way with $\mathbb{R}$. We endow each $\Omega^{i}$ with
the inner product
\begin{equation}
\left\langle f,g\right\rangle =\sum_{\sigma\in X^{i}}f\left(\sigma\right)g\left(\sigma\right)\label{eq:inner-prod}
\end{equation}
(note that $f\left(\sigma\right)g\left(\sigma\right)$ is well defined
even without choosing an orientation for $\sigma$).

For a cell $\sigma$ (either oriented or non-oriented) and a vertex
$v$, we write $v\sim\sigma$ if $v\notin\sigma$ and $\left\{ v\right\} \cup\sigma$
is a cell in $X$ (here we ignore the orientation of $\sigma$). If
$\sigma=\left[\sigma_{0},\ldots,\sigma_{j}\right]$ is oriented and
$v\sim\sigma$, then $v\sigma$ denotes the oriented $\left(j+1\right)$-cell
$\left[v,\sigma_{0},\ldots,\sigma_{j}\right]$. An oriented $\left(j+1\right)$-cell
$\left[\sigma_{0},\ldots,\sigma_{j}\right]$ induces orientations
on the $j$-cells which form its boundary, as follows: the face $\left\{ \sigma_{0},\ldots,\sigma_{i-1},\sigma_{i+1},\ldots,\sigma_{j}\right\} $
is oriented as $\left(-1\right)^{i}\left[\sigma_{1},\ldots,\sigma_{i-1},\sigma_{i+1},\ldots,\sigma_{k}\right]$,
where $\left(-1\right)\tau=\overline{\tau}$.

The \emph{$j^{\mathrm{th}}$ boundary} \emph{operator }$\partial_{j}:\Omega^{j}\rightarrow\Omega^{j-1}$
is 
\[
\left(\partial_{j}f\right)\left(\sigma\right)=\sum\limits _{v\sim\sigma}f\left(v\sigma\right).
\]
The sequence $\left(\Omega^{j},\partial_{j}\right)$ is a chain complex,
i.e., $\partial_{j-1}\partial_{j}=0$ for all $j$, and one denotes
\begin{alignat*}{2}
Z_{j} & =\ker\partial_{j} &  & j\!-\!\mathrm{cycles}\\
B_{j} & =\im\partial_{j+1} & \qquad & j\!-\!\mathrm{boundaries}\\
H_{j} & =\nicefrac{Z_{j}}{B_{j}} &  & \mathrm{the\:}j^{\mathrm{th}}\:\mathrm{homology\: of}\: X\:\left(\mathrm{over}\:\mathbb{R}\right).
\end{alignat*}
The adjoint of $\partial_{j}$ w.r.t.\ the inner product \eqref{eq:inner-prod}
is the \emph{co-boundary operator} $\partial_{j}^{*}:\Omega^{j-1}\rightarrow\Omega^{j}$
given by 
\[
\left(\partial_{j}^{*}f\right)\left(\sigma\right)=\negthickspace\sum_{{\tau\mathrm{\, is\, in\, the}\atop \mathrm{boundary\, of\,}\sigma}}\negthickspace f\left(\tau\right)=\sum_{i=0}^{j}\left(-1\right)^{i}f\left(\sigma\backslash\sigma_{i}\right),
\]
 where $\sigma\backslash\sigma_{i}=\left[\sigma_{0},\sigma_{1},\ldots,\sigma_{i-1},\sigma_{i+1},\ldots\sigma_{j}\right]$.
Here the standard terms are
\begin{alignat*}{2}
Z^{j} & =\ker\partial_{j+1}^{*}=B_{j}^{\bot} & \qquad & \mathrm{closed\:}j\!-\!\mathrm{forms}\\
B^{j} & =\im\partial_{j}^{*}=Z_{j}^{\bot} &  & \mathrm{exact\:}j\!-\!\mathrm{forms}\\
H^{j} & =\nicefrac{Z^{j}}{B^{j}} &  & \mathrm{the\:}j^{\mathrm{th}}\:\mathrm{cohomology\: of}\: X\:\left(\mathrm{over}\:\mathbb{R}\right).
\end{alignat*}
The \emph{upper, lower, }and\emph{ full Laplacians} $\Delta^{+},\Delta^{-},\Delta:\Omega^{d-1}\rightarrow\Omega^{d-1}$
are defined by
\[
\Delta^{+}=\partial_{{d}}\partial_{{d}}^{*},\qquad\quad\Delta^{-}=\partial_{{d-1}}^{*}\partial_{{d-1}},\qquad\mathrm{and}\qquad\Delta=\Delta^{+}+\Delta^{-},
\]
respectively%
\footnote{\ More generally, one can define the \emph{$j^{\mathrm{th}}$ lower
Laplacian} $\Delta_{j}^{-}:\Omega^{j}\rightarrow\Omega^{j}$ by $\Delta_{j}^{-}=\partial_{j}^{*}\partial_{j}$,
and similarly for $\Delta_{j}^{+}$ and $\Delta_{j}$. For our purposes,
$\Delta_{d-1}^{-}$, $\Delta_{d-1}^{+}$ and $\Delta_{d-1}$ are the
relevant ones.%
}. All the Laplacians decompose (as a direct sum of linear operators)
with respect to the orthogonal decompositions $\Omega^{d-1}=B^{d-1}\oplus Z_{d-1}=B_{d-1}\oplus Z^{d-1}$.
In addition, $\ker\Delta^{+}=Z^{d-1}$ and $\ker\Delta^{-}=Z_{d-1}$. 

The space of \emph{harmonic $\left(d-1\right)$-forms} on $X$ is
$\mathcal{H}_{d-1}=\ker\Delta$. If $f\in\mathcal{H}_{d-1}$ then
\[
0=\left\langle \Delta f,f\right\rangle =\left\langle \partial_{d-1}f,\partial_{d-1}f\right\rangle +\left\langle \partial_{d}^{*}f,\partial_{d}^{*}f\right\rangle 
\]
which shows that $\mathcal{H}_{d-1}=Z^{d-1}\cap Z_{d-1}$. This gives
the so-called \emph{discrete Hodge decomposition 
\[
\Omega^{d-1}=B^{d-1}\oplus\mathcal{H}_{d-1}\oplus B_{d-1}.
\]
}In particular, it follows that the space of harmonic forms can be
identified with the cohomology of $X$:
\[
H^{d-1}=\frac{Z^{d-1}}{B^{d-1}}=\frac{B_{d-1}^{\bot}}{B^{d-1}}=\frac{B^{d-1}\oplus\mathcal{H}_{d-1}}{B^{d-1}}\cong\mathcal{H}_{d-1}.
\]
The same holds for the homology of $X$, giving 
\begin{equation}
H^{d-1}\cong\mathcal{H}_{d-1}\cong H_{d-1}.\label{eq:hodge}
\end{equation}
For comparison, the original Hodge decomposition states that for a
Riemannian manifold $M$ and $0\leq j\leq\dim M$, there is an orthogonal
decomposition 
\[
\Omega^{j}\left(M\right)=d\left(\Omega^{j-1}\left(M\right)\right)\oplus\mathcal{H}^{j}\left(M\right)\oplus\delta\left(\Omega^{j+1}\left(M\right)\right)
\]
where $\Omega^{j}$ are the smooth $j$-forms on $M$, $d$ is the
exterior derivative, $\delta$ its Hodge dual, and $\mathcal{H}^{j}$
the smooth harmonic $j$-forms on $M$. As in the discrete case, this
gives an isomorphism between the $j^{\mathrm{th}}$ de-Rham cohomology
of $M$ and the space of harmonic $j$-forms on it.
\begin{example*}
For $j=0$, $Z^{0}$ consists of the locally constant functions (functions
constant on connected components); $B^{0}$ consists of the constant
functions; $Z_{0}$ of the functions whose sum vanishes, and $B_{0}$
of the functions whose sum on each connected component vanishes. 

For $j=1$, $Z^{1}$ are the forms whose sum along the boundary of
every triangle in the complex vanishes; in $B^{1}$ lie the forms
whose sum along every closed path vanishes; $Z_{1}$ are the \emph{Kirchhoff
forms}, also known as \emph{flows}, those for which the sum over all
edges incident to a vertex, oriented inward, is zero; and $B_{1}$
are the forms spanned (over $\mathbb{R}$) by oriented boundaries
of triangles in the complex. The chain of simplicial forms in dimensions
$-1$ to $2$ is depicted in Figure \ref{fig:lowermost-part}.

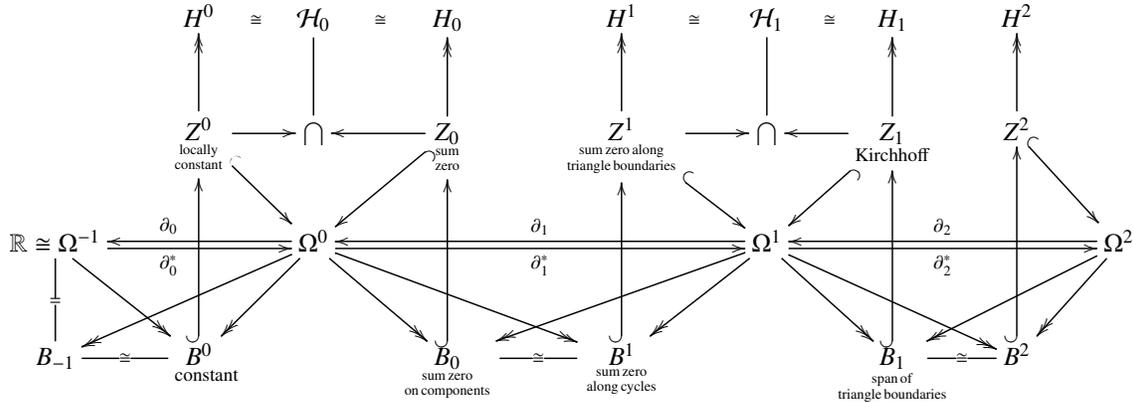
\begin{figure}[H]
\centering{}\scalebox{0.9}{$\xymatrix{ & H^{0}\ar@{}[r]|\cong & \mathcal{H}_{0}\ar@{}[r]|\cong & H_{0} & H^{1}\ar@{}[r]|\cong & \mathcal{H}_{1}\ar@{}[r]|\cong & H_{1} & H^{2}\\
 & \underset{{\mathrm{locally}\atop \mathrm{constant}}}{Z^{0}}\ar@{>>}[u]\ar@{^{(}->}[dr]\ar@{->}[r] & \bigcap\ar@{-}[u] & \underset{{\mathrm{sum}\atop \mathrm{zero}}}{Z_{0}}\ar@{>>}[u]\ar@{->}[l]\ar@{^{(}->}[dl] & \underset{{\mathrm{sum\, zero\, along}\atop \mathrm{triangle\, boundaries}}}{Z^{1}}\ar@{>>}[u]\ar@{^{(}->}[dr]\ar@{->}[r] & \bigcap\ar@{-}[u] & \underset{\mathrm{Kirchhoff}}{Z_{1}}\ar@{>>}[u]\ar@{->}[l]\ar@{^{(}->}[dl] & Z^{2}\ar@{>>}[u]\ar@{^{(}->}[dr]\\
\mathbb{R}\cong\Omega^{-1}\ar@{-}[d]|=\ar@<-1.5pt>[rr]_{\partial_{0}^{*}\quad\;}\ar@{->>}[dr] &  & \Omega^{0}\ar@{>>}[dl]\ar@{>>}[dr]\ar@{>>}[drr]\ar@{>>}[dll]\ar@<-1.5pt>[ll]_{\partial_{0}\quad\;}\ar@<-1.5pt>[rrr]_{\partial_{1}^{*}} &  &  & \Omega^{1}\ar@{>>}[dl]\ar@{>>}[dr]\ar@{>>}[dll]\ar@{>>}[drr]\ar@<-1.5pt>[lll]_{\partial_{1}}\ar@<-1.5pt>[rrr]_{\partial_{2}^{*}} &  &  & \Omega^{2}\ar@{>>}[dl]\ar@{>>}[dll]\ar@<-1.5pt>[lll]_{\partial_{2}}\\
B_{-1} & \underset{\mathrm{constant}\negthickspace\negthickspace\negthickspace}{B^{0}}\ar@{^{(}->}[uu]\ar@{-}[l]|-\cong &  & \underset{{\mathrm{sum\, zero}\atop \mathrm{on\, components}}}{B_{0}}\ar@{^{(}->}[uu]\ar@{-}[r]|-\cong & \underset{{\mathrm{sum\, zero}\atop \mathrm{along\, cycles}}}{B^{1}}\ar@{^{(}->}[uu] &  & \negthickspace\negthickspace\negthickspace\negthickspace\underset{{{\mathrm{span\, of}}\atop \mathrm{triangle\, boundaries}}}{B_{1}}\negthickspace\negthickspace\negthickspace\negthickspace\ar@{^{(}->}[uu]\ar@{-}[r]|-\cong & B^{2}\ar@{^{(}->}[uu]
}
$}\caption{\label{fig:lowermost-part}The lowermost part of the chain complex
of simplicial forms.}
\end{figure}

\end{example*}

\subsection{\label{sub:gap-def}Definition of the spectral gap}

Every graph has a ``trivial zero'' in the spectrum of its upper
Laplacian, corresponding to the constant functions. There can be more
zeros in the spectrum, and these encode information about the graph
(its connectedness), while the first one does not. Similarly, for
a $d$-dimensional complex, the space $B^{d-1}$ is always in the
kernel of the upper Laplacian, and considered to be its ``trivial
zeros''. The existence of more zeros indicates a nontrivial $\left(d-1\right)$-cohomology,
since it means that $B^{d-1}\subsetneq\ker\Delta^{+}=Z^{d-1}$. As
$\left(B^{d-1}\right)^{\bot}=Z_{d-1}$, this leads to the following
definition:
\begin{defn}
\label{def:The-spectral-gap}The spectral gap of a $d$-dimensional
complex $X$, denoted $\lambda\left(X\right)$, is the minimal eigenvalue
of the upper or the full Laplacian on $\left(d-1\right)$-cycles:
\[
\lambda\left(X\right)=\min\Spec\left(\Delta\big|_{Z_{{d-1}}}\right)=\min\Spec\left(\Delta^{+}\big|_{Z_{{d-1}}}\right)
\]
(the equality follows from $\Delta\big|_{Z_{d-1}}\equiv\Delta^{+}\big|_{Z_{d-1}}$.)
\end{defn}
The following proposition gives two more characterizations of the
spectral gap. For complexes with a complete skeleton we shall obtain
even more explicit characterizations in Proposition \ref{prop:spectral-gap-complete}.
\begin{prop}
\label{prop:gap-alternat}Let $\Spec\Delta^{+}=\left\{ \lambda_{0}\leq\lambda_{1}\leq\ldots\leq\lambda_{\left|X^{d-1}\right|-1}\right\} $.
\begin{enumerate}
\item If $\beta_{j}=\dim H_{j}$ is the $j^{\mathrm{th}}$ (reduced) Betti
number of $X$, then 
\[
\lambda\left(X\right)=\lambda_{r}\qquad\mathrm{where}\qquad r=\left(\left|X^{d-1}\right|-\beta_{d-1}\right)-\left(\left|X^{d}\right|-\beta_{d}\right).
\]

\item $\lambda\left(X\right)$ is the minimal nonzero eigenvalue of $\Delta^{+}$,
unless $X$ has a nontrivial $\left(d-1\right)^{\mathrm{th}}$-homology,
in which case $\lambda\left(X\right)=0$.
\end{enumerate}
\end{prop}
$ $

\vspace{-15pt}

\begin{rem*}
For a graph $G=\left(V,E\right)$, Definition \ref{def:The-spectral-gap}
states that $\lambda\left(G\right)$ is the minimal eigenvalue of
the Laplacian on a function which sums to zero. By Proposition \ref{prop:gap-alternat}
$\left(1\right)$ we have $\lambda\left(G\right)=\lambda_{r}$, where
$r=n-\left|E\right|-\beta_{0}+\beta_{1}$. Since $\beta_{0}+1$ is
the number of connected components in $G$, and $\beta_{1}$ is the
number of cycles in $G$, by Euler's formula 
\[
r=n-\left|E\right|-\beta_{0}+\beta_{1}=\chi\left(G\right)-\left(\chi\left(G\right)-1\right)=1
\]
and therefore $\lambda\left(G\right)=\lambda_{1}$. From $\left(2\right)$
in Proposition \ref{prop:gap-alternat} we obtain that $\lambda\left(G\right)$
is the minimal nonzero eigenvalue of $G$'s Laplacian if $G$ is connected,
and zero otherwise.\end{rem*}
\begin{proof}
Since $\Delta^{+}$ decomposes w.r.t.\ $\Omega^{d-1}=B^{{d-1}}\oplus Z_{{d-1}}$,
and $\Delta^{+}\big|_{B^{d-1}}\equiv0$, the spectrum of $\Delta^{+}$
consists of $r=\dim B^{d-1}$ zeros, followed by the spectral gap.
By \eqref{eq:hodge}, 
\[
H_{d-1}\cong\mathcal{H}_{d-1}=Z^{d-1}\cap Z_{d-1}=\ker\Delta^{+}\big|_{Z_{d-1}}
\]
so that $\lambda\left(X\right)=0$ if and only if $H_{d-1}\neq0$,
i.e.\ $X$ has a nontrivial $\left(d-1\right)^{\mathrm{th}}$-homology.
This also shows that if $H_{d-1}=0$, then $\lambda\left(X\right)$
is the smallest nonzero eigenvalue of $\Delta^{+}$. Finally, to compute
$r=\dim B^{d-1}$, we observe that 
\begin{align*}
\dim B^{j-1} & =\dim Z^{j-1}-\dim H^{j-1}=\nul\partial_{j}^{*}-\beta_{j-1}\\
 & =\dim\Omega^{j-1}-\rank\partial_{j}^{*}-\beta_{j-1}=\left|X^{j-1}\right|-\dim B^{j}-\beta_{j-1}
\end{align*}
and therefore
\begin{align*}
r & =\dim B^{d-1}=\left|X^{d-1}\right|-\dim B^{d}-\beta_{d-1}=\left|X^{d-1}\right|-\left(\left|X^{d}\right|-\dim B^{d+1}-\beta_{d}\right)-\beta_{d-1}\\
 & =\left(\left|X^{d-1}\right|-\beta_{d-1}\right)-\left(\left|X^{d}\right|-\beta_{d}\right).
\end{align*}
\vspace{-33pt}

\end{proof}

\section{Properties of the Laplacians }

In this section we begin the study of the Laplacians and their spectra.
We start by writing the Laplacians in a more explicit form. 

For the upper Laplacian, if $f\in\Omega^{d-1}$ and $\sigma\in X^{d-1}$,
then 
\begin{align}
\left(\Delta^{+}f\right)\left(\sigma\right) & =\sum_{v\sim\sigma}\left(\partial_{{d-1}}^{*}f\right)\left(v\sigma\right)=\sum_{v\sim\sigma}\sum_{i=0}^{{d}}\left(-1\right)^{i}f\left(v\sigma\backslash\left(v\sigma\right)_{i}\right)\nonumber \\
 & =\sum_{v\sim\sigma}f\left(\sigma\right)-\sum_{i=0}^{{d-1}}\left(-1\right)^{i}f\left(v\sigma\backslash\sigma_{i}\right)\nonumber \\
 & =\deg\left(\sigma\right)f\left(\sigma\right)-\sum_{v\sim\sigma}\sum_{i=0}^{{d-1}}\left(-1\right)^{i}f\left(v\sigma\backslash\sigma_{i}\right),\label{eq:upper-laplacian}
\end{align}
where we recall that $\deg\left(\sigma\right)$ is the number of $d$-cells
containing $\sigma$. Let us introduce the following notation: for
$\sigma,\sigma'\in X_{\pm}^{d-1}$ we denote $\sigma'\sim\sigma$
if there exists an oriented $d$-cell $\tau$ such that both $\sigma$
and $\overline{\sigma'}$ are in the boundary of $\tau$ (as oriented
cells). Using this notation we can express $\Delta^{+}$ more elegantly
as 
\begin{equation}
\left(\Delta^{+}f\right)\left(\sigma\right)=\deg\left(\sigma\right)f\left(\sigma\right)-\sum_{\sigma'\sim\sigma}f\left(\sigma'\right).\label{eq:upper-lap-neighb}
\end{equation}
For the lower Laplacian we have

\begin{equation}
\left(\Delta^{-}f\right)\left(\sigma\right)=\sum_{i=0}^{{d-1}}\left(-1\right)^{i}\left(\partial_{{d-1}}f\right)\left(\sigma\backslash\sigma_{i}\right)=\sum_{i=0}^{{d-1}}\left(-1\right)^{i}\sum_{v\sim\sigma\backslash\sigma_{i}}f\left(v\sigma\backslash\sigma_{i}\right).\label{eq:lower-laplacian}
\end{equation}

The following straightforward claim bounds the spectrum of the upper
Laplacian:
\begin{claim}
The spectrum of $\Delta^{+}$ is contained in the interval $\left[0,\left(d+1\right)k\right]$,
where $k$ is the maximal degree in $X$.\hfill{}\qed
\end{claim}

\subsection{Complexes with a complete skeleton}

Complexes with a complete skeleton appear to be particularly well
behaved, in comparison with the general case. The following proposition
lists some observations regarding their Laplacians. These will be
used in the proofs of the main theorems, and also to obtain simpler
characterizations of the spectral gap in this case.
\begin{prop}
\label{prop:full-skel-prop}If $X$ has a complete skeleton, then
\begin{enumerate}
\item If $\overline{X}$ is the complement complex of $X$, i.e., $\overline{X}^{{d-1}}=X^{{d-1}}={V \choose d}$%
\footnote{\ ${V \choose j}$ denotes the set of subsets of $V$ of size $j$.%
} and $\overline{X}^{{d}}={V \choose d+1}\backslash X^{{d}}$, then
\begin{equation}
\Delta_{\overline{X}}^{+}=n\cdot I-\Delta_{X}.\label{eq:laplacian-complement}
\end{equation}

\item The spectrum of $\Delta$ lies in the interval $\left[0,n\right]$.
\item The lower Laplacian of $X$ satisfies
\begin{equation}
\Delta^{-}=n\cdot\mathbb{P}_{B^{{d-1}}}\label{eq:lower-laplace-proj}
\end{equation}
where $\mathbb{P}_{B^{{d-1}}}$ is the orthogonal projection onto
$B^{{d-1}}$.
\end{enumerate}
\end{prop}
\begin{proof}
By the completeness of the skeleton, the lower Laplacian (see \eqref{eq:lower-laplacian})
can be written as 
\begin{align*}
\left(\Delta^{-}f\right)\left(\sigma\right) & =\sum_{i=0}^{{d-1}}\left(-1\right)^{i}\sum_{v\sim\sigma\backslash\sigma_{i}}f\left(v\sigma\backslash\sigma_{i}\right)=\sum_{i=0}^{{d-1}}\left(-1\right)^{i}\sum_{v\notin\sigma\backslash\sigma_{i}}f\left(v\sigma\backslash\sigma_{i}\right)\\
 & =d\cdot f\left(\sigma\right)+\sum_{v\notin\sigma}\sum_{i=0}^{{d-1}}\left(-1\right)^{i}f\left(v\sigma\backslash\sigma_{i}\right).
\end{align*}
To show $\mathit{\left(1\right)}$ we observe that $v\sim\sigma$
in $\overline{X}$ iff $v\notin\sigma$ and $v\nsim\sigma$ (in $X$),
so that 
\begin{align*}
\left(\Delta_{X}f+\Delta_{\overline{X}}^{+}f\right)\left(\sigma\right)= & \left(\Delta_{X}^{-}f\right)\left(\sigma\right)+\left(\Delta_{X}^{+}f\right)\left(\sigma\right)+\left(\Delta_{\overline{X}}^{+}f\right)\left(\sigma\right)\\
= & \, d\cdot f\left(\sigma\right)+\sum_{v\notin\sigma}\sum_{i=0}^{{d-1}}\left(-1\right)^{i}f\left(v\sigma\backslash\sigma_{i}\right)\\
 & +\deg\left(\sigma\right)f\left(\sigma\right)-\sum_{v\sim\sigma}\sum_{i=0}^{{d-1}}\left(-1\right)^{i}f\left(v\sigma\backslash\sigma_{i}\right)\\
 & +\left(n-d-\deg\left(\sigma\right)\right)f\left(\sigma\right)-\sum_{{v\notin\sigma\atop v\nsim\sigma}}\sum_{i=0}^{{d-1}}\left(-1\right)^{i}f\left(v\sigma\backslash\sigma_{i}\right)=nf\left(\sigma\right).
\end{align*}
From $\mathit{\left(1\right)}$ we conclude that $\Spec\Delta_{\overline{X}}^{+}=\left\{ n-\gamma\,\big|\,\gamma\in\Spec\Delta_{X}\right\} $,
and since $\Delta_{X}$ and $\Delta_{\overline{X}}^{+}$ are positive
semidefinite, $\mathit{\left(2\right)}$ follows. To establish $\mathit{\left(3\right)}$,
recall that $\left(B^{{d-1}}\right)^{\bot}=Z_{d-1}=\ker\Delta^{-},$
and it is left to show that $\Delta^{-}f=nf$ for $f\in B^{{d-1}}$.
Note that $B^{d-1}\subseteq Z^{d-1}=\ker\Delta_{X}^{+}$, and in addition,
that since $B^{{d-1}}$ only depends on $X$'s $\left(d-1\right)$-skeleton,
\[
B^{{d-1}}\left(X\right)=B^{{d-1}}\left(\overline{X}\right)\subseteq Z^{{d-1}}\left(\overline{X}\right)=\ker\Delta_{\overline{X}}^{+}.
\]
Now from $\mathit{\left(1\right)}$ it follows that for $f\in B^{d-1}$
\[
\Delta_{X}^{-}f=\Delta_{X}^{-}f+\Delta_{X}^{+}f=\Delta_{X}f=nf-\Delta_{\overline{X}}^{+}f=nf
\]
as desired.
\end{proof}
The next proposition offers alternative characterizations of the spectral
gap:
\begin{prop}
\label{prop:spectral-gap-complete}If $X$ has a complete skeleton,
then
\begin{enumerate}
\item \textup{The spectral gap of $X$ is obtained by
\begin{equation}
\lambda\left(X\right)=\min\Spec\,\Delta.\label{eq:spectral_gap_for_full_skel}
\end{equation}
}
\item \emph{Furthermore, it is the ${n-1 \choose d-1}+1$ smallest eigenvalue
of $\Delta^{+}$.}
\end{enumerate}
\end{prop}
$ $\begin{rems*}$ $
\begin{enumerate}
\item For graphs \eqref{eq:spectral_gap_for_full_skel} gives $\lambda\left(G\right)=\min\Spec\left(\Delta^{+}+J\right)$,
where $J=\Delta^{-}=\left(\begin{smallmatrix}1 & 1 & \cdots & 1\\
1 & 1 & \cdots & 1\\
\vdots & \vdots & \ddots & \vdots\\
1 & 1 & \cdots & 1
\end{smallmatrix}\right)$.
\item In general \eqref{eq:spectral_gap_for_full_skel} does not hold: for
example, for the triangle complex $\blacktriangleright\!\blacktriangleleft$,
$\lambda=\min\Spec\left(\Delta\big|_{Z_{{d-1}}}\right)=3$ but $\min\Spec\,\Delta=1$.
\end{enumerate}
\end{rems*}
\begin{proof}
$ $
\begin{enumerate}
\item First, since $\Delta$ decomposes w.r.t.\ $\Omega^{d-1}=B^{{d-1}}\oplus Z_{{d-1}}$
we have 
\[
\Spec\,\Delta=\Spec\Delta\big|_{B^{{d-1}}}\,\cup\,\Spec\Delta\big|_{Z_{{d-1}}}=\Spec\Delta^{-}\big|_{B^{{d-1}}}\,\cup\,\Spec\Delta^{+}\big|_{Z_{{d-1}}}.
\]
By Proposition \ref{prop:full-skel-prop}, $\Spec\Delta^{-}\big|_{B^{{d-1}}}=\left\{ n\right\} $
and $\Spec\,\Delta\subseteq\left[0,n\right]$, which implies that
\[
\lambda=\min\Spec\left(\Delta^{+}\big|_{Z_{{d-1}}}\right)=\min\Spec\,\Delta.
\]

\item The Euler characteristic satisfies $\sum_{i=-1}^{d}\left(-1\right)^{i}\left|X^{i}\right|=\chi\left(X\right)=\sum_{i=-1}^{d}\left(-1\right)^{i}\beta_{i}$.
Therefore, by Proposition \ref{prop:gap-alternat} we have $\lambda=\lambda_{r}$,
with 
\begin{align*}
r & =\left(\left|X^{d-1}\right|-\beta_{d-1}\right)-\left(\left|X^{d}\right|-\beta_{d}\right)\\
 & =\left(\left|X^{d-1}\right|-\beta_{d-1}\right)-\left(\left|X^{d}\right|-\beta_{d}\right)+\left(-1\right)^{d}\sum_{i=-1}^{d}\left(-1\right)^{i}\left(\left|X^{i}\right|-\beta_{i}\right)\\
 & =\sum_{i=-1}^{d-2}\left(-1\right)^{d+i}\left(\left|X^{i}\right|-\beta_{i}\right).
\end{align*}
Since the $\left(d-1\right)$-skeleton is complete, $\left|X^{i}\right|={n \choose i+1}$
and $\beta_{i}=0$ for $0\leq i\leq d-2$, and so 
\[
r=\sum_{i=-1}^{d-2}\left(-1\right)^{d+i}{n \choose i+1}={n-1 \choose d-1}.
\]
\vspace{-33pt}

\end{enumerate}
\end{proof}
We finish with a note on the density of $d$-cells in $X$:
\begin{prop}
\label{prop:density-degree-spec}Let $\delta$ denote the $d$-cell
density of $X$, $\delta=\frac{\left|X^{d}\right|}{{n \choose d+1}}$,
let $k$ denote the average degree of a $\left(d-1\right)$-cell,
and let $\lambda_{avg}$ denote the average over the spectrum of $\Delta^{+}\big|_{Z_{d-1}}$.
Then
\[
\delta=\frac{\lambda_{avg}}{n}=\frac{k}{n-d}.
\]
\end{prop}
\begin{proof}
On the one hand
\[
\delta=\frac{\left|X^{d}\right|}{{n \choose d+1}}=\frac{\left|X^{d-1}\right|\frac{k}{d+1}}{{n \choose d+1}}=\frac{{n \choose d}\frac{k}{d+1}}{{n \choose d+1}}=\frac{k}{n-d}.
\]
On the other,
\[
{n \choose d}k=\left|X^{d-1}\right|k=\sum_{\sigma\in X^{d-1}}\deg\sigma=\tr\Delta^{+}=\sum_{\lambda\in\Spec\Delta^{+}}\lambda=\negthickspace\sum_{\lambda\in\Spec\Delta^{+}|_{Z_{d-1}}}\negthickspace\negthickspace\lambda
\]
and by Proposition \ref{prop:spectral-gap-complete} 
\[
\lambda_{avg}=\frac{1}{{n \choose d}-{n-1 \choose d-1}}\sum_{\lambda\in\Spec\Delta^{+}|_{Z_{d-1}}}\negthickspace\lambda=\frac{1}{{n-1 \choose d}}\sum_{\lambda\in\Spec\Delta^{+}|_{Z_{d-1}}}\negthickspace\lambda=\frac{n}{n-d}\cdot k.
\]
\vspace{-25pt}

\end{proof}

\section{Proofs of the main theorems}

\subsection{A Cheeger-type inequality}

This section is devoted to the proof of Theorem \ref{thm:Isoperimetric_inequality}:
For a complex with a complete skeleton, the Cheeger constant is bounded
from below by the spectral gap.
\begin{proof}[Proof of Theorem \ref{thm:Isoperimetric_inequality}]
Recall that we seek to show 
\[
\min\Spec\left(\Delta^{+}\big|_{Z_{{d-1}}}\right)=\lambda\left(X\right)\leq h\left(X\right)=\min\limits _{V=\coprod_{i=0}^{d}A_{i}}\frac{n\cdot\left|F\left(A_{0},A_{1},\ldots,A_{d}\right)\right|}{\left|A_{0}\right|\cdot\left|A_{1}\right|\cdot\ldots\cdot\left|A_{d}\right|}.
\]
Let $A_{0},\ldots,A_{{d}}$ be a partition of $V$ which realizes
the minimum in $h$. We define $f\in\Omega^{d-1}$ by 
\begin{equation}
f\left(\left[\sigma_{0}\:\sigma_{1}\:\ldots\:\sigma_{{d-1}}\right]\right)=\begin{cases}
\sgn\left(\pi\right)\left|A_{\pi\left({d}\right)}\right| & \exists\pi\in\mathrm{Sym}_{\left\{ 0\ldots{d}\right\} }\:\mathrm{with}\:\sigma_{i}\in A_{\pi\left(i\right)}\:\mathrm{for}\:0\leq i\leq{d-1}\\
0 & \mathrm{else,\: i.e.\:}\exists k,i\neq j\:\mathrm{with}\:\sigma_{i},\sigma_{j}\in A_{k}.
\end{cases}\label{eq:cheeger_to_lambda_f}
\end{equation}
Note that $f\left(\pi'\sigma\right)=\sgn\left(\pi'\right)f\left(\sigma\right)$
for any $\pi'\in\mathrm{Sym}_{\left\{ 0\ldots d-1\right\} }$ and
$\sigma\in X^{d-1}$. Therefore, $f$ is a well-defined skew-symmetric
function on oriented $\left(d-1\right)$-cells, i.e., $f\in\Omega^{d-1}$.
Figure \ref{fig:The-function-f} illustrates $f$ for $d=1,2$.

\begin{figure}[h]
\begin{centering}
\includegraphics[width=0.7\columnwidth]{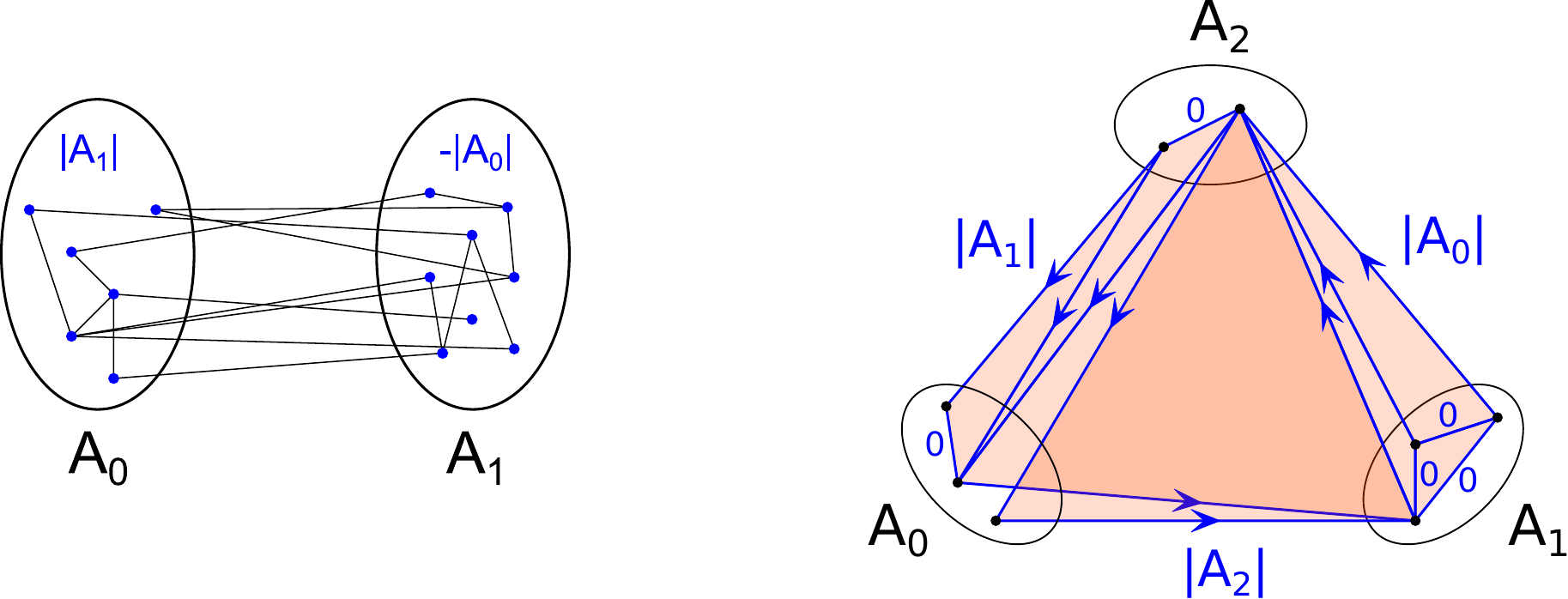}
\par\end{centering}

\caption{\label{fig:The-function-f}The form $f\in\Omega^{d-1}$ defined in
\eqref{eq:cheeger_to_lambda_f}, for complexes of dimensions one and
two.}
\end{figure}

We proceed to show that $f\in Z_{{d-1}}$. Let $\sigma=\left[\sigma_{0},\sigma_{1},\ldots,\sigma_{d-2}\right]\in X_{\pm}^{d-2}$.
As we assumed that $X^{{d-1}}$ is complete, 
\[
\left(\partial_{{d-1}}f\right)\left(\sigma\right)=\sum_{v\sim\sigma}f\left(\left[v,\sigma_{0},\sigma_{1},\ldots,\sigma_{{d-2}}\right]\right)=\sum_{v\notin\sigma}f\left(\left[v,\sigma_{0},\sigma_{1},\ldots,\sigma_{{d-2}}\right]\right).
\]
If for some $k$ and $i\neq j$ we have $\sigma_{i},\sigma_{j}\in A_{k}$,
this sum vanishes. On the other hand, if there exists $\pi\in\mathrm{Sym}_{\left\{ 0\ldots{d}\right\} }$
such that $\sigma_{i}\in A_{\pi\left(i\right)}$ for $0\leq i\leq d-2$
then 
\begin{align*}
\left(\partial_{d-1}f\right)\left(\sigma\right) & =\sum_{v\in A_{\pi\left(d-1\right)}}f\left(\left[v,\sigma_{0},\sigma_{1},\ldots,\sigma_{{d-2}}\right]\right)+\sum_{v\in A_{\pi\left(d\right)}}f\left(\left[v,\sigma_{0},\sigma_{1},\ldots,\sigma_{{d-2}}\right]\right)\\
 & =\sum_{v\in A_{\pi\left(d-1\right)}}\left(-1\right)^{d-1}\sgn\pi\left|A_{\pi\left(d\right)}\right|+\sum_{v\in A_{{d}}}\left(-1\right)^{d}\sgn\pi\left|A_{\pi\left(d-1\right)}\right|\\
 & =\left(-1\right)^{d-1}\sgn\pi\left(\left|A_{\pi\left(d-1\right)}\right|\left|A_{\pi\left(d\right)}\right|-\left|A_{\pi\left(d\right)}\right|\left|A_{\pi\left(d-1\right)}\right|\right)=0
\end{align*}
and in both cases $f\in Z_{{d-1}}$. Thus, by Rayleigh's principle
\begin{equation}
\lambda\left(X\right)=\min\Spec\left(\Delta^{+}\big|_{Z_{{d-1}}}\right)\leq\frac{\left\langle \Delta^{+}f,f\right\rangle }{\left\langle f,f\right\rangle }=\frac{\left\langle \partial_{{d}}^{*}f,\partial_{{d}}^{*}f\right\rangle }{\left\langle f,f\right\rangle }.\label{eq:Rayleigh}
\end{equation}
The denominator is 
\[
\left\langle f,f\right\rangle =\sum_{\sigma\in X^{d-1}}f\left(\sigma\right)^{2},
\]
and a $\left({d-1}\right)$-cell $\sigma$ contributes to this sum
only if its vertices are in different blocks of the partition, i.e.,
there are no $k$ and $i\neq j$ with $\sigma_{i},\sigma_{j}\in A_{k}$.
In this case, there exists a unique block, $A_{i}$, which does not
contain a vertex of $\sigma$, and $\sigma$ contributes $\left|A_{i}\right|^{2}$
to the sum. Since $X^{{d-1}}$ is complete, there are $\left|A_{0}\right|\cdot\ldots\cdot\left|A_{i-1}\right|\cdot\left|A_{i+1}\right|\cdot\ldots\cdot\left|A_{{d}}\right|$
non-oriented $\left({d-1}\right)$-cells whose vertices are in distinct
blocks and which do not intersect $A_{i}$, hence
\[
\left\langle f,f\right\rangle =\sum_{i=0}^{{d}}\left(\prod_{j\neq i}\left|A_{j}\right|\right)\left|A_{i}\right|^{2}=n\prod_{i=0}^{{d}}\left|A_{i}\right|.
\]
To evaluate the numerator in \eqref{eq:Rayleigh}, we first show that
for $\sigma\in X^{d}$ 
\begin{equation}
\left|\left(\partial_{{d}}^{*}f\right)\left(\sigma\right)\right|=\begin{cases}
n & \sigma\in F\left(A_{0},\ldots,A_{{d}}\right)\\
0 & \sigma\notin F\left(A_{0},\ldots,A_{{d}}\right)
\end{cases}.\label{eq:upper_d_on_f}
\end{equation}
First, let $\sigma\notin F\left(A_{0},\ldots,A_{{d}}\right)$. If
$\sigma$ has three vertices from the same $A_{i}$, or two pairs
of vertices from the same blocks (i.e.\ $\sigma_{i},\sigma_{j}\in A_{k}$
and $\sigma_{i'},\sigma_{j'}\in A_{k'}$), then for every summand
in 
\[
\left(\partial_{{d}}^{*}f\right)\left(\sigma\right)=\sum_{i=0}^{{d}}\left(-1\right)^{i}f\left(\sigma\backslash\sigma_{i}\right),
\]
the cell $\sigma\backslash\sigma_{i}$ has two vertices from the same
block, and therefore $\left(\partial_{{d}}^{*}f\right)\left(\sigma\right)=0$.
Next, assume that $\sigma_{j}$ and $\sigma_{k}$ (with $j<k$) is
the only pair of vertices in $\sigma$ which belong to the same block.
The only non-vanishing terms in $\left(\partial_{{d}}^{*}f\right)\left(\sigma\right)=\sum_{i=0}^{{d}}\left(-1\right)^{i}f\left(\sigma\backslash\sigma_{i}\right)$
are $i=j$ and $i=k$, i.e., 
\[
\left(\partial_{{d}}^{*}f\right)\left(\sigma\right)=\left(-1\right)^{j}f\left(\sigma\backslash\sigma_{j}\right)+\left(-1\right)^{k}f\left(\sigma\backslash\sigma_{k}\right).
\]
Since the value of $f$ on a simplex depends only on the blocks to
which its vertices belong, 
\begin{align*}
f\left(\sigma\backslash\sigma_{j}\right) & =f\left(\left[\sigma_{0}\:\sigma_{1}\:\ldots\:\sigma_{j-1}\:\sigma_{j+1}\:\ldots\sigma_{k-1}\:\sigma_{k}\:\sigma_{k+1}\ldots\:\sigma_{{d}}\right]\right)\\
 & =f\left(\left[\sigma_{0}\:\sigma_{1}\:\ldots\:\sigma_{j-1}\:\sigma_{j+1}\:\ldots\sigma_{k-1}\:\sigma_{j}\:\sigma_{k+1}\ldots\:\sigma_{{d}}\right]\right)\\
 & =f\left(\left(-1\right)^{k-j+1}\left[\sigma_{0}\:\sigma_{1}\:\ldots\:\sigma_{j-1}\:\sigma_{j}\:\sigma_{j+1}\:\ldots\:\sigma_{k-1}\:\sigma_{k+1}\ldots\:\sigma_{{d}}\right]\right)\\
 & =\left(-1\right)^{k-j+1}f\left(\sigma\backslash\sigma_{k}\right),
\end{align*}
so that
\[
\left(\partial_{{d}}^{*}f\right)\left(\sigma\right)=\left(-1\right)^{j}\left(-1\right)^{k-j+1}f\left(\sigma\backslash\sigma_{k}\right)+\left(-1\right)^{k}f\left(\sigma\backslash\sigma_{k}\right)=0.
\]
The remaining case is $\sigma\in F\left(A_{0},\ldots,A_{{d}}\right)$.
Here, there exists $\pi\in\mathrm{Sym}_{\left\{ 0\ldots{d}\right\} }$
with $\sigma_{i}\in A_{\pi\left(i\right)}$ for $0\leq i\leq{d}$.
Observe that 
\[
f\left(\sigma\backslash\sigma_{i}\right)=\sgn\left(\pi\cdot\left(d\;{d\negmedspace-\negmedspace1}\; d\negmedspace-\negmedspace2\:\ldots\: i\right)\right)\left|A_{\pi\left(i\right)}\right|=\left(-1\right)^{{d}-i}\sgn\pi\left|A_{\pi\left(i\right)}\right|
\]
and therefore
\[
\left(\partial_{{d}}^{*}f\right)\left(\sigma\right)=\sum_{i=0}^{{d}}\left(-1\right)^{i}f\left(\sigma\backslash\sigma_{i}\right)=\left(-1\right)^{{d}}\sgn\pi\sum_{i=0}^{{d}}\left|A_{\pi\left(i\right)}\right|=\left(-1\right)^{{d}}\sgn\pi n.
\]
Therefore, $\left|\left(\partial_{{d}}^{*}f\right)\left(\sigma\right)\right|=n$.
This establishes \eqref{eq:upper_d_on_f}, which implies that 
\[
\left\langle \partial_{{d}}^{*}f,\partial_{{d}}^{*}f\right\rangle =\sum_{\sigma\in X^{d}}\left|\left(\partial_{{d}}^{*}f\right)\left(\sigma\right)\right|^{2}=n^{2}\left|F\left(A_{0},\ldots,A_{{d}}\right)\right|
\]
and in total
\[
\lambda\left(X\right)\leq\frac{\left\langle \partial_{{d}}^{*}f,\partial_{{d}}^{*}f\right\rangle }{\left\langle f,f\right\rangle }=\frac{n\left|F\left(A_{0},\ldots,A_{{d}}\right)\right|}{\prod_{i=0}^{{d}}\left|A_{i}\right|}=h\left(X\right).
\]

\end{proof}

\subsection{\label{sub:Towards-a-lower}Towards a lower Cheeger inequality}

The first observation to be made regarding a lower Cheeger inequality,
is that no bound of the form $C\cdot h\left(X\right)^{m}\leq\lambda\left(X\right)$
can be found. Had such a bound existed, one would have that $\lambda\left(X\right)=0$
implies $h\left(X\right)=0$, but a counterexample to this is provided
by the minimal triangulation of the Möbius strip (Figure \ref{fig:Mobius}).

\begin{figure}[H]
\centering{}\includegraphics[scale=0.65]{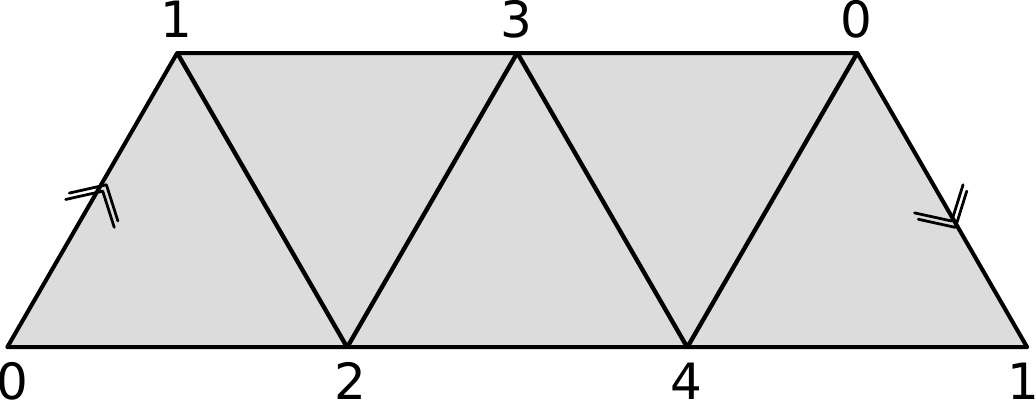}\caption{\label{fig:Mobius}A triangulation of the Möbius strip for which $h\left(X\right)=1\frac{1}{4}$
but $\lambda\left(X\right)=0$.}
\end{figure}

Nevertheless, numerical experiments hint that a bound of the form
$C\cdot h\left(X\right)^{2}-c\leq\lambda\left(X\right)$ should hold,
where $C$ and $c$ depend on the dimension and the maximal degree
of a $\left(d-1\right)$-cell in $X$.

An attempt towards an upper bound for the Cheeger constant can be
made by connecting it to ``local Cheeger constants'', as follows.
For every $\tau\in X^{d-2}$ we consider the \emph{link }of $\tau$
(see Figure \ref{fig:link}), 
\[
\lk\tau=\left\{ \sigma\in X\,\middle|\,\sigma\cap\tau=\varnothing\;\mathrm{and}\;\sigma\cup\tau\in X\right\} .
\]
\begin{figure}[h]
\begin{centering}
\includegraphics[scale=1.5]{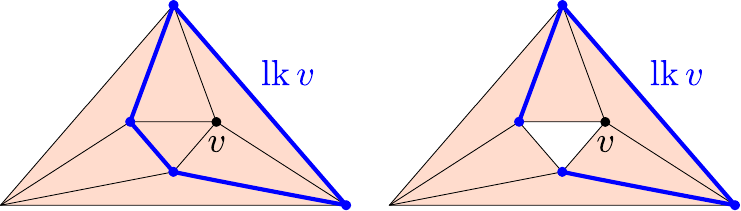}
\par\end{centering}

\caption{\label{fig:link}Two examples for the link of a vertex in a triangle
complex.}
\end{figure}

Since $\dim\tau=d-2$, $\lk\tau$ is a graph, and there is a $1-1$
correspondence between vertices (edges) of $\lk\tau$ and $\left(d-1\right)$-cells
($d$-cells) of $X$ which contain $\tau$. We have the following
bound for the Cheeger constant of $X$:
\begin{prop}
The bound $h\left(X\right)\leq\frac{h\left(\lk\tau\right)}{1-\frac{d-1}{n}}$
holds for any $d$-complex $X$ and $\tau\in X^{d-2}$.\end{prop}
\begin{proof}
Write $\tau=\left[\tau_{0},\tau_{1},\ldots,\tau_{d-2}\right]$ and
denote $A_{i}=\left\{ \tau_{i}\right\} $ for $0\leq i\leq d-2$.
Due to the correspondence between $\left(\lk\tau\right)^{j}$ and
cells in $X^{d-1+j}$ containing $\tau$, 
\[
h\left(\lk\tau\right)\overset{{\scriptscriptstyle def}}{=}\min_{B\coprod C=\left(\lk\tau\right)^{0}}\frac{\left|E_{\lk\tau}\left(B,C\right)\right|\cdot\left|\left(\lk\tau\right)^{0}\right|}{\left|B\right|\cdot\left|C\right|}=\min_{B\coprod C=\left(\lk\tau\right)^{0}}\frac{\left|F\left(A_{0},\ldots,A_{d-2},B,C\right)\right|\cdot\left|\left(\lk\tau\right)^{0}\right|}{\left|B\right|\cdot\left|C\right|}.
\]
Assume that the minimum is attained by $B=B_{0}$ and $C=C_{0}$.
We define 
\[
A_{d-1}=B_{0},\qquad A_{d}=V\backslash\left(\bigcup_{i=0}^{d-1}A_{i}\right).
\]
Now $A_{0},\ldots,A_{d}$ is a partition of $V$, and 
\[
F\left(A_{0},\ldots,A_{d-2},B_{0},C_{0}\right)=F\left(A_{0},\ldots,A_{d-2},A_{d-1},A_{d}\right)
\]
since no $d$-cell containing $\tau$ has a vertex in $A_{d}\backslash C_{0}$.
In addition, 
\begin{align*}
\frac{\left|\left(\lk\tau\right)^{0}\right|\left|A_{d}\right|}{n\left|C_{0}\right|} & \geq\frac{\left|\left(\lk\tau\right)^{0}\right|\left|A_{d}\right|-\left|A_{d-1}\right|\left(\left|A_{d}\right|-\left|C_{0}\right|\right)}{n\left|C_{0}\right|}\\
 & =\frac{\left[n-\left(d-1\right)-\left(\left|A_{d}\right|-\left|C_{0}\right|\right)\right]\left|A_{d}\right|-\left|A_{d-1}\right|\left(\left|A_{d}\right|-\left|C_{0}\right|\right)}{n\left|C_{0}\right|}\\
 & =\frac{\left(n-\left(d-1\right)\right)\left|A_{d}\right|-\left(\left|A_{d-1}\right|+\left|A_{d}\right|\right)\left(\left|A_{d}\right|-\left|C_{0}\right|\right)}{n\left|C_{0}\right|}\\
 & =\frac{\left(n-\left(d-1\right)\right)\left[\left|A_{d}\right|-\left(\left|A_{d}\right|-\left|C_{0}\right|\right)\right]}{n\left|C_{0}\right|}=1-\frac{d-1}{n},
\end{align*}
which implies 
\begin{align*}
h\left(\lk\tau\right) & =\frac{F\left(A_{0},\ldots,A_{d-2},A_{d-1},A_{d}\right)\left|\left(\lk\tau\right)^{0}\right|}{\left|B_{0}\right|\cdot\left|C_{0}\right|}\\
 & =\frac{F\left(A_{0},\ldots,A_{d-2},A_{d-1},A_{d}\right)n}{\left|A_{0}\right|\cdot\ldots\cdot\left|A_{d}\right|}\cdot\frac{\left|\left(\lk\tau\right)^{0}\right|\left|A_{d}\right|}{n\left|C_{0}\right|}\\
 & \geq h\left(X\right)\cdot\frac{\left|\left(\lk\tau\right)^{0}\right|\left|A_{d}\right|}{n\left|C_{0}\right|}\geq\left(1-\frac{d-1}{n}\right)h\left(X\right).
\end{align*}
\vspace{-25pt}

\end{proof}
Since $\lk\tau$ is a graph, its Cheeger constant can be bounded using
the lower inequality in \eqref{eq:cheeger-graph}. We also note that
the degree of a vertex in $\lk\tau$ corresponds to the degree of
a $\left(d-1\right)$-cell in $X$, and therefore 
\begin{equation}
\frac{\left(1-\frac{d-1}{n}\right)^{2}}{8k}h^{2}\left(X\right)\leq\frac{h\left(\lk\tau\right)^{2}}{8k}\leq\frac{h\left(\lk\tau\right)^{2}}{8k_{\tau}}\leq\lambda\left(\lk\tau\right)\label{eq:from_h_to_lambda_lk}
\end{equation}
where $k$ is the maximal degree of a $\left(d-1\right)$-cell in
$X$, and $k_{\tau}$ of a vertex in $\lk\tau$.

We now see that a bound of the spectral gap of links by that of the
complex would yield a lower Cheeger inequality. Such a bound was indeed
discovered by Garland in \cite{Gar73}, and was studied further by
several authors \cite{Zuk96,ABM05,GW12}. The following lemma appears
in \cite{GW12}, for a normalized version of the Laplacian. We give
here, without proof, its form for the Laplacian we use.
\begin{lem}[\cite{Gar73,GW12}]
\label{lem:link-laplacian}Let $X$ be a $d$-dimensional simplicial
complex. Given $f\in\Omega^{d-1},\sigma\in X^{d-1},\tau\in X^{d-2}$
define a function $f_{\tau}:\left(\lk\tau\right)^{0}\rightarrow\mathbb{R}$
by $f_{\tau}\left(v\right)=f\left(v\tau\right)$, and an operator
$\Delta_{\tau}^{+}:\Omega^{d-1}\left(X\right)\rightarrow\Omega^{d-1}\left(X\right)$
by
\[
\left(\Delta_{\tau}^{+}f\right)\left(\sigma\right)=\begin{cases}
\deg_{\tau}\left(\sigma\right)f\left(\sigma\right)-\smash{\sum\limits _{{\sigma'\sim\sigma\atop \tau\subseteq\sigma'}}}f\left(\sigma'\right)\vphantom{\underset{|}{A}} & \quad\tau\subset\sigma\\
0 & \quad\tau\nsubseteq\sigma
\end{cases}
\]
where $\deg_{\tau}\left(\sigma\right)=\#\left\{ \sigma'\sim\sigma\,\middle|\,\tau\subseteq\sigma'\right\} =\deg_{\lk\tau}\left(\sigma\backslash\tau\right)$.
The following then hold:
\begin{enumerate}
\item $\Delta^{+}=\left(\sum_{\tau\in X^{d-2}}\Delta_{\tau}^{+}\right)-\left(d-1\right)D$,
where \textup{$\left(Df\right)\left(\sigma\right)=\deg\left(\sigma\right)f\left(\sigma\right)$.}
\item $\left\langle \Delta_{\tau}^{+}f,f\right\rangle =\left\langle \Delta_{\lk\tau}^{+}f_{\tau},f_{\tau}\right\rangle $.
\item If $f\in Z_{d-1}$ then $f_{\tau}\in Z_{0}\left(\lk\tau\right)$.
\item $\sum_{\tau\in X^{d-2}}\left\langle f_{\tau},f_{\tau}\right\rangle =d\left\langle f,f\right\rangle $.
\end{enumerate}
\end{lem}
Assume now that $f\in Z_{d-1}$ is a normalized eigenfunction for
$\lambda\left(X\right)$, i.e.\ $\left\langle f,f\right\rangle =1$
and $\Delta^{+}f=\lambda\left(X\right)f$. Using the lemma we find
that 
\begin{multline*}
\lambda\left(X\right)=\left\langle \Delta^{+}f,f\right\rangle \overset{\mathit{\left(1\right)}}{=}\sum_{\tau\in X^{d-2}}\left\langle \Delta_{\tau}^{+}f,f\right\rangle -\left(d-1\right)\left\langle Df,f\right\rangle \overset{\mathit{\left(2\right)}}{=}\sum_{\tau\in X^{d-2}}\left\langle \Delta_{\lk\tau}^{+}f_{\tau},f_{\tau}\right\rangle -\left(d-1\right)\left\langle Df,f\right\rangle \\
\geq\sum_{\tau\in X^{d-2}}\left\langle \Delta_{\lk\tau}^{+}f_{\tau},f_{\tau}\right\rangle -\left(d-1\right)k\overset{\mathit{\left(3\right)}}{\geq}\sum_{\tau\in X^{d-2}}\lambda\left(\lk\tau\right)\left\langle f_{\tau},f_{\tau}\right\rangle -\left(d-1\right)k\overset{\mathit{\left(4\right)}}{=}d\min_{\tau\in X^{d-2}}\lambda\left(\lk\tau\right)-\left(d-1\right)k.
\end{multline*}
By \eqref{eq:from_h_to_lambda_lk} we obtain the bound
\[
\frac{d\left(1-\frac{d-1}{n}\right)^{2}}{8k}h^{2}\left(X\right)-\left(d-1\right)k\leq\lambda\left(X\right).
\]
Sadly, this bound is trivial, as it is not hard to show that the l.h.s.\ is
non-positive for every complex $X$. A possible line of research would
be to find a stronger relation between the spectral gap of the complex
and that of its links, for the case of complexes with a complete skeleton
(Garland's work applies to general ones).

\subsection{\label{sub:The-Mixing-Lemma}The Mixing Lemma}

Here we prove Theorem \ref{thm:mixing}. We begin by formulating it
precisely. 
\begin{thm*}[\ref{thm:mixing}]
Let $X$ be a $d$-dimensional complex with a complete skeleton.
Fix $\alpha\in\mathbb{R}$, and write $\Spec\left(\alpha I-\Delta^{+}\right)=\left\{ \mu_{0}\geq\mu_{1}\geq\ldots\geq\mu_{m}\right\} $
(where $m={n \choose d}-1$). For any disjoint sets of vertices $A_{0},\ldots,A_{d}$
(not necessarily a partition), one has 
\[
\left|\left|F\left(A_{0},\ldots,A_{d}\right)\right|-\frac{\alpha\cdot\left|A_{0}\right|\cdot\ldots\cdot\left|A_{d}\right|}{n}\right|\leq\rho_{\alpha}\cdot\left(\left|A_{0}\right|\cdot\ldots\cdot\left|A_{d}\right|\right)^{\frac{d}{d+1}}
\]
where 
\[
\rho_{\alpha}=\max\left\{ \big|\mu_{{n-1 \choose d-1}}\big|,\left|\mu_{m}\right|\right\} =\left\Vert \left(\alpha I-\Delta^{+}\right)\big|_{Z_{d-1}}\right\Vert .
\]
\end{thm*}
\begin{rem}
\label{rem:discrepancy}Which $\alpha$ should one take in practice?
In the introduction we state the theorem for $\alpha=k$, the average
degree of a $\left(d-1\right)$-cell, so that it generalize the familiar
form of the Expander Mixing Lemma for $k$-regular graphs. However,
the expectation of $\left|F\left(A_{0},\ldots,A_{d}\right)\right|$
in a random settings is actually $\delta\left|A_{0}\right|\cdot\ldots\cdot\left|A_{d}\right|$,
where $\delta$ is the $d$-cell density $\frac{\left|X^{d}\right|}{{n \choose d}}$.
Therefore, $\alpha=n\delta=\frac{nk}{n-d}$ is actually a more accurate
choice. This becomes even clearer upon observing that we seek to minimize
$\rho_{\alpha}=\left\Vert \left(\alpha I-\Delta^{+}\right)\big|_{Z_{d-1}}\right\Vert $,
since Proposition \ref{prop:density-degree-spec} shows that the spectrum
of $\Delta^{+}\big|_{Z_{d-1}}$ is centered around $\lambda_{avg}=n\delta=\frac{nk}{n-d}$.
While for a fixed $d$ the choice between $k$ and $\frac{nk}{n-d}$
is negligible, this should be taken into account when $d$ depends
on $n$.\end{rem}
\begin{proof}
For any disjoint sets of vertices $A_{0},\ldots,A_{d-1}$, define
$\delta_{A_{0},\ldots,A_{d-1}}\in\Omega^{d-1}$ by
\[
\delta_{A_{0},\ldots,A_{d-1}}\left(\sigma\right)=\begin{cases}
\sgn\left(\pi\right) & \exists\pi\in\mathrm{Sym}_{\left\{ 0\ldots d-1\right\} }\:\mathrm{with}\:\sigma_{i}\in A_{\pi\left(i\right)}\:\mathrm{for}\:0\leq i\leq{d-1}\\
0 & \mathrm{else}
\end{cases}.
\]
Since the skeleton of $X$ is complete, 
\begin{equation}
\left\Vert \delta_{A_{0},\ldots,A_{d-1}}\right\Vert =\sqrt{\sum_{\sigma\in X^{d-1}}\delta_{A_{0},\ldots,A_{d-1}}^{2}\left(\sigma\right)}=\sqrt{\left|A_{0}\right|\cdot\ldots\cdot\left|A_{d-1}\right|}.\label{eq:norm-of-charact}
\end{equation}
Now, let $A_{0},\ldots,A_{d}$ be disjoint subsets of $V$ (not necessarily
a partition), and denote
\begin{align*}
\varphi & =\delta_{A_{0},A_{1},A_{2},\ldots,A_{d-1}}\\
\psi & =\delta_{A_{d},A_{1},A_{2},\ldots,A_{d-1}}.
\end{align*}
Let $\sigma$ be an oriented $\left(d-1\right)$-cell with one vertex
in each of $A_{0},A_{1},\ldots,A_{d-1}$. We shall denote this by
$\sigma\in F\left(A_{0},\ldots,A_{d-1}\right)$, ignoring the orientation
of $\sigma$. There is a correspondence between $d$-cells in $F\left(A_{0},\ldots,A_{d}\right)$
containing $\sigma$, and neighbors of $\sigma$ which lie in $F\left(A_{d},A_{1},\ldots,A_{d-1}\right)$.
Furthermore, for such a neighbor $\sigma'$ we have $\varphi\left(\sigma\right)=\psi\left(\sigma'\right)$,
since $\sigma$ and $\sigma'$ must share the vertices which belong
to $A_{1},\ldots,A_{d-1}$. Therefore, if $\left(Df\right)\left(\sigma\right)=\deg\left(\sigma\right)f\left(\sigma\right)$
then by \eqref{eq:upper-lap-neighb} 
\begin{align}
\left\langle \varphi,\left(D-\Delta^{+}\right)\psi\right\rangle  & =\sum_{\sigma\in X^{d-1}}\varphi\left(\sigma\right)\left(\left(D-\Delta^{+}\right)\psi\right)\left(\sigma\right)=\sum_{\sigma\in X^{d-1}}\sum_{\sigma'\sim\sigma}\varphi\left(\sigma\right)\psi\left(\sigma'\right)\nonumber \\
 & =\sum_{\sigma\in F\left(A_{0}\ldots A_{d-1}\right)}\sum_{\sigma'\sim\sigma}\varphi\left(\sigma\right)\psi\left(\sigma'\right)=\sum_{\sigma\in F\left(A_{0}\ldots A_{d-1}\right)}\#\left\{ \sigma'\in F\left(A_{d},A_{1},\ldots,A_{d-1}\right)\,\middle|\,\sigma'\sim\sigma\right\} \nonumber \\
 & =\sum_{\sigma\in F\left(A_{0}\ldots A_{d-1}\right)}\#\left\{ \tau\in F\left(A_{0},A_{1},\ldots,A_{d}\right)\,\middle|\,\sigma\subseteq\tau\right\} =\left|F\left(A_{0},A_{1},\ldots,A_{d}\right)\right|.\label{eq:long_calc}
\end{align}
Notice that since the $A_{i}$ are disjoint, $\varphi$ and $\psi$
are supported on different $\left(d-1\right)$-cells, so that for
any $\alpha\in\mathbb{R}$ 
\begin{equation}
\left\langle \varphi,\left(D-\Delta^{+}\right)\psi\right\rangle =\left\langle \varphi,-\Delta^{+}\psi\right\rangle =\left\langle \varphi,\left(\alpha I-\Delta^{+}\right)\psi\right\rangle .\label{eq:disjoint-support}
\end{equation}
As $\Delta^{+}$ decomposes w.r.t.\ the orthogonal decomposition
$\Omega^{d-1}=B^{d-1}\oplus Z_{d-1}$, and since $B^{d-1}\subseteq Z^{d-1}=\ker\Delta^{+}$,
\begin{align}
\left|F\left(A_{0},A_{1},\ldots,A_{d}\right)\right| & =\left\langle \varphi,\left(\alpha I-\Delta^{+}\right)\psi\right\rangle \nonumber \\
 & =\left\langle \varphi,\left(\alpha I-\Delta^{+}\right)\left(\mathbb{P}_{B^{d-1}}\psi+\mathbb{P}_{Z_{d-1}}\psi\right)\right\rangle \nonumber \\
 & =\left\langle \varphi,\alpha\mathbb{P}_{B^{d-1}}\psi+\left(\alpha I-\Delta^{+}\right)\mathbb{P}_{Z_{d-1}}\psi\right\rangle \nonumber \\
 & =\alpha\left\langle \varphi,\mathbb{P}_{B^{d-1}}\psi\right\rangle +\left\langle \varphi,\left(\alpha I-\Delta^{+}\right)\mathbb{P}_{Z_{d-1}}\psi\right\rangle .\label{eq:F_by_B_and_Z}
\end{align}
We proceed to evaluate each of these terms separately. Using \eqref{eq:lower-laplace-proj}
and \eqref{eq:laplacian-complement} we find that 
\[
\alpha\left\langle \varphi,\mathbb{P}_{B^{d-1}}\psi\right\rangle =\frac{\alpha}{n}\left\langle \varphi,\Delta^{-}\psi\right\rangle =\frac{\alpha}{n}\left\langle \varphi,\left(nI-\Delta_{X}^{+}-\Delta_{\overline{X}}^{+}\right)\psi\right\rangle 
\]
and by \eqref{eq:long_calc} and \eqref{eq:disjoint-support} this
implies
\begin{align}
\alpha\left\langle \varphi,\mathbb{P}_{B^{d-1}}\psi\right\rangle  & =\frac{\alpha}{n}\left\langle \varphi,\left(nI-\Delta_{X}^{+}\right)\psi\right\rangle +\frac{\alpha}{n}\left\langle \varphi,-\Delta_{\overline{X}}^{+}\psi\right\rangle \nonumber \\
 & =\frac{\alpha}{n}\left|F_{X}\left(A_{0},A_{1},\ldots,A_{d}\right)\right|+\frac{\alpha}{n}\left|F_{\overline{X}}\left(A_{0},A_{1},\ldots,A_{d}\right)\right|\nonumber \\
 & =\frac{\alpha\cdot\left|A_{0}\right|\cdot\ldots\cdot\left|A_{d}\right|}{n}.\label{eq:P_B}
\end{align}
We turn to the second term in \eqref{eq:F_by_B_and_Z}. First, we
recall from Proposition \ref{prop:spectral-gap-complete} that $\dim B^{d-1}={n-1 \choose d-1}$.
Since $B^{d-1}\subseteq\ker\Delta^{+}$, we can assume that in $\Spec\left(\alpha I-\Delta^{+}\right)=\left\{ \mu_{0}\geq\mu_{1}\geq\ldots\geq\mu_{m}\right\} $
the first ${n-1 \choose d-1}$ values correspond to $B^{d-1}$, and
the rest to $\left(B^{d-1}\right)^{\bot}=Z_{d-1}$. Thus,
\begin{equation}
\rho_{\alpha}=\max\left\{ \big|\mu_{{n-1 \choose d-1}}\big|,\left|\mu_{m}\right|\right\} =\max\left\{ \left|\mu\right|\,\middle|\,\mu\in\Spec\left(\alpha I-\Delta^{+}\right)\big|_{Z_{d-1}}\right\} =\left\Vert \left(\alpha I-\Delta^{+}\right)\big|_{Z_{d-1}}\right\Vert ,\label{eq:rho_alpha_def}
\end{equation}
and therefore 
\begin{align}
\left|\left\langle \varphi,\left(\alpha I-\Delta^{+}\right)\mathbb{P}_{Z_{d-1}}\psi\right\rangle \right| & \leq\left\Vert \varphi\right\Vert \cdot\left\Vert \left(\alpha I-\Delta^{+}\right)\mathbb{P}_{Z_{d-1}}\psi\right\Vert \leq\left\Vert \varphi\right\Vert \cdot\left\Vert \left(\alpha I-\Delta^{+}\right)\big|_{Z_{d-1}}\right\Vert \cdot\left\Vert \mathbb{P}_{Z_{d-1}}\psi\right\Vert \nonumber \\
 & \leq\rho_{\alpha}\cdot\left\Vert \varphi\right\Vert \cdot\left\Vert \psi\right\Vert =\rho_{\alpha}\sqrt{\left|A_{0}\right|\left|A_{d}\right|}\left|A_{1}\right|\left|A_{2}\right|\ldots\left|A_{d-1}\right|,\label{eq:P_Z}
\end{align}
where the last step is by \eqref{eq:norm-of-charact}. Together \eqref{eq:F_by_B_and_Z},
\eqref{eq:P_B} and \eqref{eq:P_Z} give 
\[
\left|\left|F\left(A_{0},A_{1},\ldots,A_{d}\right)\right|-\frac{\alpha\cdot\left|A_{0}\right|\cdot\ldots\cdot\left|A_{d}\right|}{n}\right|\leq\rho_{\alpha}\sqrt{\left|A_{0}\right|\left|A_{d}\right|}\left|A_{1}\right|\left|A_{2}\right|\ldots\left|A_{d-1}\right|.
\]
Since $A_{0},\ldots,A_{d}$ play the same role, one can also obtain
the bound 
\[
\rho_{\alpha}\sqrt{\left|A_{\pi\left(0\right)}\right|\left|A_{\pi\left(d\right)}\right|}\left|A_{\pi\left(1\right)}\right|\left|A_{\pi\left(2\right)}\right|\ldots\left|A_{\pi\left(d-1\right)}\right|,
\]
for any $\pi\in\mathrm{Sym}_{\left\{ 0..d\right\} }$. Taking the
geometric mean over all such $\pi$ gives
\[
\left|\left|F\left(A_{0},A_{1},\ldots,A_{d}\right)\right|-\frac{\alpha\cdot\left|A_{0}\right|\cdot\ldots\cdot\left|A_{d}\right|}{n}\right|\leq\rho_{\alpha}\cdot\left(\left|A_{0}\right|\left|A_{1}\right|\ldots\left|A_{d}\right|\right)^{\frac{d}{d+1}}.
\]
\end{proof}
\begin{rem*}
The estimate \eqref{eq:P_Z} is somewhat wasteful. As is done in graphs,
a slightly better one is 
\begin{align*}
\left|\left\langle \varphi,\left(\alpha I-\Delta^{+}\right)\mathbb{P}_{Z_{d-1}}\psi\right\rangle \right| & =\left|\left\langle \mathbb{P}_{Z_{d-1}}\varphi,\left(\alpha I-\Delta^{+}\right)\mathbb{P}_{Z_{d-1}}\psi\right\rangle \right|\leq\rho_{\alpha}\cdot\left\Vert \mathbb{P}_{Z_{d-1}}\varphi\right\Vert \cdot\left\Vert \mathbb{P}_{Z_{d-1}}\psi\right\Vert ,
\end{align*}
and we leave it to the curious reader to verify that this gives
\[
\left|\left\langle \varphi,\left(\alpha I-\Delta^{+}\right)\mathbb{P}_{Z_{d-1}}\psi\right\rangle \right|\leq\rho_{\alpha}\sqrt{\left|A_{0}\right|\left(1-\frac{\sum_{i=0}^{d-1}\left|A_{i}\right|}{n}\right)\left|A_{d}\right|\left(1-\frac{\sum_{i=1}^{d}\left|A_{i}\right|}{n}\right)}\left|A_{1}\right|\ldots\left|A_{d-1}\right|.
\]

\end{rem*}

\subsection{\label{sub:Gromov's-Geometric-Overlap}Gromov's geometric overlap}

Here we prove Corollary \ref{cor:Geometric_overlap_cor}, which gives
a bound on the geometric overlap of a complex in terms of the width
of its spectrum.
\begin{proof}[Proof of Corollary \ref{cor:Geometric_overlap_cor}]
Given $\varphi:V\rightarrow\mathbb{R}^{d+1}$, choose arbitrarily
some partition of $V$ into equally sized parts $P_{0},\ldots,P_{d}$.
By Pach's theorem \cite{Pac98}, there exist $c_{d}>0$ and $Q_{i}\subseteq P_{i}$
of sizes $\left|Q_{i}\right|=c_{d}\left|P_{i}\right|$ such that for
some $x\in\mathbb{R}^{d+1}$ we have $x\in\mathrm{conv}\left\{ \varphi\left(v\right)\,\middle|\, v\in\sigma\right\} $
for any $\sigma\in F\left(Q_{0},\ldots,Q_{d}\right)$. By the Mixing
Lemma (Theorem \ref{thm:mixing}),
\[
\left|F\left(Q_{0},\ldots,Q_{d}\right)\right|\geq\frac{k\cdot\left|Q_{0}\right|\cdot\ldots\cdot\left|Q_{d}\right|}{n}-\varepsilon\cdot\left(\left|Q_{0}\right|\cdot\ldots\cdot\left|Q_{d}\right|\right)^{\frac{d}{d+1}}=\left(\frac{c_{d}n}{d+1}\right)^{d}\left(\frac{kc_{d}}{d+1}-\varepsilon\right).
\]
On the other hand,
\[
\left|X^{d}\right|=\left|X^{d-1}\right|\frac{k}{d+1}={n \choose d}\frac{k}{d+1}\leq\left(\frac{en}{d}\right)^{d}\frac{k}{d+1}.
\]
As this holds for every $\varphi$,
\[
\overlap\left(X\right)\geq\left(\frac{c_{d}d}{e\left(d+1\right)}\right)^{d}\left(c_{d}-\frac{\varepsilon\left(d+1\right)}{k}\right)\geq\frac{c_{d}^{d}}{e^{d+1}}\left(c_{d}-\frac{\varepsilon\left(d+1\right)}{k}\right).
\]
 \end{proof}
\begin{rem}
\label{rem:overlap-lambda}Following Remark \ref{rem:discrepancy},
if $\Spec\Delta^{+}\big|_{Z_{d-1}}\subseteq\left[\lambda_{avg}-\varepsilon',\lambda_{avg}+\varepsilon'\right]$
then using the Mixing Lemma with $\alpha=\lambda_{avg}=\frac{nk}{n-d}$
one has 
\[
\left|F\left(Q_{0},\ldots,Q_{d}\right)\right|\geq\frac{k\cdot\left|Q_{0}\right|\cdot\ldots\cdot\left|Q_{d}\right|}{n-d}-\varepsilon'\cdot\left(\left|Q_{0}\right|\cdot\ldots\cdot\left|Q_{d}\right|\right)^{\frac{d}{d+1}}\geq\left(\frac{c_{d}n}{d+1}\right)^{d}\left(\frac{nkc_{d}}{\left(n-d\right)\left(d+1\right)}-\varepsilon'\right)
\]
so that 
\[
\overlap\left(X\right)\geq\frac{c_{d}^{d}n}{e^{d+1}\left(n-d\right)}\left(c_{d}-\frac{\varepsilon'\left(d+1\right)}{\lambda_{avg}}\right).
\]

\end{rem}

\subsection{\label{sub:Expansion-in-random}Expansion in random complexes}

In this section we prove Corollaries \ref{cor:Random_cor} and \ref{cor:random-overlap},
regarding the expansion of random Linial-Meshulam complexes. The main
idea is the following lemma, which is a variation on the analysis
in \cite{GW12} of the spectrum of $D-\Delta^{+}$ for $X=X\left(d,n,p\right)$.
\begin{lem}
\label{lem:spec_bound}Let $c>0$. There exists $\gamma=O\left(\negmedspace\sqrt{C}\right)$
such that $X=X\left(d,n,\frac{C\cdot\log n}{n}\right)$ satisfies
\textup{
\[
\Spec\left(\Delta^{+}\big|_{Z_{d-1}}\right)\subseteq\left[\left(C-\gamma\right)\log n,\left(C+\gamma\right)\log n\right]
\]
with probability at least $1-n^{-c}$.}\end{lem}
\begin{proof}
We denote $p=\frac{C\cdot\log n}{n}$. For $C$ large enough we shall
find $\gamma=O\left(\negmedspace\sqrt{C}\right)$ such that 
\begin{equation}
\left\Vert \left(\Delta^{+}-pn\cdot I\right)\big|_{Z_{d-1}}\right\Vert \leq\gamma\log n\label{eq:close-to-complete}
\end{equation}
holds with probability at least $1-n^{-c}$. This implies the Lemma,
as 
\[
\Spec\left(\Delta^{+}\big|_{Z_{d-1}}\right)\subseteq\left[pn-\gamma\log n,pn+\gamma\log n\right]=\left[\left(C-\gamma\right)\log n,\left(C+\gamma\right)\log n\right].
\]
To show \eqref{eq:close-to-complete} we use
\begin{align}
\left\Vert \left(\Delta^{+}-pn\cdot I\right)\big|_{Z_{d-1}}\right\Vert  & =\left\Vert \left(\Delta^{+}-p\left(n-d\right)I-pdI+D-D\right)\big|_{Z_{d-1}}\right\Vert \nonumber \\
 & \leq\left\Vert \left(D-p\left(n-d\right)I\right)\big|_{Z_{d-1}}\right\Vert +\left\Vert \left(D-\Delta^{+}+pdI\right)\big|_{Z_{d-1}}\right\Vert \label{eq:big-triangle-ineq}
\end{align}
and we will treat each term separately. For the first, we have
\[
\left\Vert \left(D-\left(n-d\right)pI\right)\big|_{Z_{d-1}}\right\Vert \leq\left\Vert D-\left(n-d\right)pI\right\Vert =\max_{\sigma\in X^{d-1}}\left|\deg\sigma-\left(n-d\right)p\right|.
\]
Since $\deg\sigma\sim B\left(n-d,p\right)$, a Chernoff type bound
(e.g. \cite[Theorem 1]{Jan02}) gives that for every $t>0$
\[
\mathrm{Prob}\left(\left|\deg\sigma-\left(n-d\right)p\right|>t\right)\leq2e^{-\frac{t^{2}}{2\left(n-d\right)p+\frac{2t}{3}}}.
\]
By a union bound on the degrees of the $\left(d-1\right)$-cells we
get 
\begin{equation}
\mathrm{Prob}\left(\max_{\sigma\in X^{d-1}}\left|\deg\sigma-\left(n-d\right)p\right|>t\right)\leq2{n \choose d}e^{-\frac{t^{2}}{2\left(n-d\right)p+\frac{2t}{3}}},\label{eq:union-deg}
\end{equation}
and a straightforward calculation shows that there exists $\alpha=\alpha\left(c,d\right)>0$
such that for $t=\alpha\sqrt{np\log n}$, the r.h.s.\ in \eqref{eq:union-deg}
is bounded by $\frac{1}{2n^{c}}$ for large enough $C$ and $n$.
In total this implies 
\begin{equation}
\mathrm{Prob}\left(\left\Vert \left(D-\left(n-d\right)pI\right)\big|_{Z_{d-1}}\right\Vert \leq\alpha\sqrt{C}\log n\right)\geq1-\frac{1}{2n^{c}}.\label{eq:degree-deviation}
\end{equation}
In order to understand the last term in \eqref{eq:big-triangle-ineq}
we follow \cite{GW12}, which shows that $\left(D_{X}-\Delta_{X}^{+}\right)\big|_{Z_{d-1}}$
is close to $p$ times $\left(D_{K_{n}^{d}}-\Delta_{K_{n}^{d}}^{+}\right)\big|_{Z_{d-1}}$$ $,
where $K_{n}^{d}$ is the complete $d$-complex on $n$ vertices.
Note that $D_{K_{n}^{d}}=\left(n-d\right)\cdot I$ and $\Delta_{K_{n}^{d}}^{+}\big|_{Z_{d-1}}=n\cdot I$,
and that $Z_{d-1}\left(X\right)=Z_{d-1}\left(K_{n}^{d}\right)$ as
both have the same $\left(d-1\right)$-skeleton. In the proof of Theorem
$7$ in \cite{GW12} (which uses an idea from \cite{Oli10}), it is
shown that 
\[
\mathrm{Prob}\left(\left\Vert \left(D_{X}-\Delta_{X}^{+}+pdI\right)\big|_{Z_{d-1}}\right\Vert \geq t\right)=\mathrm{Prob}\left(\left\Vert \left(D_{X}-\Delta_{X}^{+}\right)\big|_{Z_{d-1}}-p\left(D_{K_{n}^{d}}-\Delta_{K_{n}^{d}}^{+}\right)\big|_{Z_{d-1}}\right\Vert \geq t\right)\leq2{n \choose d}e^{-\frac{t^{2}}{8pnd+4t}}.
\]
Again, there exists $\beta=\beta\left(c,d\right)>0$ such that for
$t=\beta\sqrt{np\log n}$, the r.h.s.\ is bounded by $\frac{1}{2n^{c}}$
for large enough $C$ and $n$. Consequently, 
\[
\mathrm{Prob}\left(\left\Vert \left(D-\Delta^{+}+pdI\right)\big|_{Z_{d-1}}\right\Vert \leq\beta\sqrt{C}\log n\right)\geq1-\frac{1}{2n^{c}},
\]
so that 
\[
\mathrm{Prob}\left(\left\Vert \left(\Delta^{+}-pnI\right)\big|_{Z_{d-1}}\right\Vert \leq\left(\alpha+\beta\right)\sqrt{C}\log n\right)\geq1-n^{-c},
\]
and $\gamma=\left(\alpha+\beta\right)\sqrt{C}$ gives the required
result.
\end{proof}
We obtain the following corollary, which implies in particular Corollaries
\ref{cor:Random_cor} and \ref{cor:random-overlap}.
\begin{cor}
\label{cor:Random_unified}Observe $X=X\left(d,n,\frac{C\cdot\log n}{n}\right)$. 
\begin{enumerate}
\item Given $c>0$, there exist a constant $H=C-O\left(\sqrt{C}\right)$
such that for large enough $n$ 
\begin{equation}
\mathrm{Prob}\left(h\left(X\right)\geq H\cdot\log n\right)\geq1-n^{-c},\label{eq:cheeger-random}
\end{equation}
and for any $\vartheta<\left(\frac{c_{d}}{e}\right)^{d+1}$ (where
$c_{d}$ is Pach's constant \cite{Pac98}), for $C$ and $n$ large
enough 
\[
\mathrm{Prob}\left(\overlap\left(X\right)>\vartheta\right)\geq1-n^{-c}.
\]

\item If $C<1$ then $\mathrm{Prob}\left(h\left(X\right)=0\right)\overset{{\scriptscriptstyle n\rightarrow\infty}}{\longrightarrow}1.$ 
\end{enumerate}
\end{cor}
\begin{proof}
$\left(1\right)$ Since $\lambda\left(X\right)\leq h\left(X\right)$
(Theorem \ref{thm:Isoperimetric_inequality}), \eqref{eq:cheeger-random}
follows from Lemma \ref{lem:spec_bound} with $H=C-\gamma$ (recall
that $\gamma=O\left(\negmedspace\sqrt{C}\right)$). We turn to the
geometric overlap. From Lemma \ref{lem:spec_bound} it follows that
for $C$ large enough a.a.s.\ $\Spec\Delta^{+}\big|_{Z_{d-1}}\subseteq\left[\left(C-\gamma\right)\log n,\left(C+\gamma\right)\log n\right]$.
Therefore, $\Spec\left(\Delta^{+}\big|_{Z_{d-1}}\right)\subseteq\left[\lambda_{avg}-\varepsilon',\lambda_{avg}+\varepsilon'\right]$
with $\varepsilon'=2\gamma\log n$. By Remark \ref{rem:overlap-lambda},
\[
\overlap\left(X\right)\geq\frac{c_{d}^{d}n}{e^{d+1}\left(n-d\right)}\left(c_{d}-\frac{2\gamma\log n\left(d+1\right)}{\lambda_{avg}}\right)\geq\frac{c_{d}^{d}}{e^{d+1}}\left(c_{d}-\frac{2\gamma\left(d+1\right)}{C-\gamma}\right)\overset{{\scriptscriptstyle C\rightarrow\infty}}{\longrightarrow}\left(\frac{c_{d}}{e}\right)^{d+1}.
\]
$\left(2\right)$ Choose some $\tau\in X^{d-2}$. It was observed
in \cite{GW12} that $\lk\tau\sim G\left(n-d+1,\frac{C\cdot\log n}{n}\right)$
(where $G\left(n,p\right)=X\left(1,n,p\right)$ is the Erd\H{o}s\textendash{}Rényi
model), and $G\left(n,\frac{C\cdot\log n}{n}\right)$ has isolated
vertices a.a.s.\ for $C<1$ \cite{erdHos1959random,erdos1961evolution}.
These correspond to isolated $\left(d-1\right)$-cells in $X$ (cells
of degree zero), whose existence implies $h\left(X\right)=0$ (and
thus also $\lambda\left(X\right)=0$).
\end{proof}

\section{\label{sec:Open-Questions}Open questions }

\paragraph{Non-complete skeleton.}

The proof of the generalized mixing lemma assumes that the skeleton
is complete. This raises the following question:
\begin{lyxlist}{00.00.0000}
\item [{\textbf{Question:}}] Can the discrepancy in $X$ be bounded for
general simplicial complexes?
\end{lyxlist}
As remarked after the statement of Theorem \ref{thm:Isoperimetric_inequality},
one always has $h\left(X\right)=0$ for $X$ with a non-complete skeleton.
This calls for a refined definition, and a natural candidate is the
following:
\[
\widetilde{h}\left(X\right)=\min\limits _{V=\coprod_{i=0}^{d}A_{i}}\frac{n\cdot\left|F\left(A_{0},A_{1},\ldots,A_{d}\right)\right|}{\left|F^{\partial}\left(A_{0},A_{1},\ldots,A_{d}\right)\right|},
\]
where $F^{\partial}\left(A_{0},A_{1},\ldots,A_{d}\right)$ denotes
the set of $\left(d-1\right)$-spheres (i.e.\ copies of the $\left(d-1\right)$-skeleton
of the $d$-simplex) having one vertex in each $A_{i}$. For a complex
$X$ with a complete skeleton, $\widetilde{h}\left(X\right)=h\left(X\right)$
as $F^{\partial}\left(A_{0},\ldots,A_{d}\right)=A_{0}\times\ldots\times A_{d}$.
It is not hard to see that a lower Cheeger inequality does not hold
here: consider any non-minimal triangulation of the $\left(d-1\right)$-shpere,
and attach a single $d$-simplex to one of the $\left(d-1\right)$-cells
on it. The obtained complex has $\lambda=0$, and $\widetilde{h}=n$.
However, we conjecture that the upper bound still holds:
\begin{lyxlist}{00.00.0000}
\item [{\textbf{Question:}}] Does the inequality $\lambda\left(X\right)\leq\widetilde{h}\left(X\right)$
holds for every $d$-complex?
\end{lyxlist}

\paragraph{Inverse Mixing Lemma}

In \cite{BL06} Bilu and Linial prove an Inverse Mixing Lemma for
graphs:
\begin{thm*}[\cite{BL06}]
Let $G$ be a $k$-regular graph on $n$ vertices. Suppose that for
any disjoint $A,B\subseteq V$
\[
\left|E\left(A,B\right)-\frac{k\left|A\right|\left|B\right|}{n}\right|\leq\rho\sqrt{\left|A\right|\left|B\right|}.
\]
Then the nontrivial eigenvalues of $kI-\Delta_{G}^{+}$ are bounded,
in absolute value, by $O\left(\rho\left(1+\log\left(\frac{k}{\rho}\right)\right)\right)$.\end{thm*}
\begin{lyxlist}{00.00.0000}
\item [{\textbf{Question:}}] \noindent Can one prove a generalized Inverse
Mixing Lemma for simplicial complexes?
\end{lyxlist}

\paragraph{Random simplicial complexes}

In the random graph model $G=G\left(n,p\right)=X\left(1,n,p\right)$,
taking $p=\frac{k}{n}$ with a fixed $k$ gives disconnected $G$
a.a.s. However, random $k$-regular graphs are a.a.s.\ connected,
and in fact are excellent expanders (see e.g.\ \cite{friedman2003proof,puder2012expansion}).
In higher dimension, $X=X\left(d,n,\frac{k}{n}\right)$ has a.a.s.\ a
nontrivial $\left(d-1\right)$-homology, and also $h\left(X\right)=0$
(by Corollary \ref{cor:Random_unified} $\left(2\right)$). It is
thus natural to ask about the expansion quality of $k$-regular $d$-complexes,
but since it is not clear whether such complexes even exist, we say
that a $k$\emph{-semiregular }complex is a complex with $k-\sqrt{k}\leq\deg\sigma\leq k+\sqrt{k}$
for all $\sigma\in X^{\dim X-1}$, and ask:
\begin{lyxlist}{00.00.0000}
\item [{\textbf{Question:}}] \noindent Are $\lambda\left(X\right)$, $h\left(X\right)$
and $\overlap\left(X\right)$ bounded away from zero with high probability,
for $X$ a random $k$-semiregular $d-$complex with a complete skeleton?
\end{lyxlist}

\paragraph{A Riemannian analogue}

In Riemannian geometry, the Cheeger constant of a Riemannian manifold
$M$ is concerned with its partitions into two submanifolds along
a common boundary of codimension one. The original Cheeger inequalities,
due to Cheeger \cite{Che70} and Buser \cite{Bus82}, relate the Cheeger
constant to the smallest eigenvalue of the Laplace-Beltrami operator
on $C^{\infty}\left(M\right)=\Omega^{0}\left(M\right)$.
\begin{lyxlist}{00.00.0000}
\item [{\textbf{Question:}}] Can one define an isoperimetric quantity which
concerns partitioning of $M$ into $d+1$ parts, and relate it to
the spectrum of the Laplace-Beltrami operator on $\Omega^{d-1}\left(M\right)$,
the space of smooth $\left(d-1\right)$-forms?
\end{lyxlist}

\paragraph{Ramanujan complexes}

\emph{Ramanujan Graphs }are expanders which are spectrally optimal
in the sense of the Alon-Boppana theorem \cite{Nil91}, and therefore
excellent combinatorial expanders. Such graphs were constructed in
\cite{LPS88} as quotients of the Bruhat-Tits tree associated with
$\mathrm{PSL}_{2}\left(\mathbb{Q}_{p}\right)$ by certain arithmetic
lattices. Analogue quotients of the Bruhat-Tits buildings associated
with $\mathrm{PSL}_{d}\left(\mathbb{F}_{q}\left(\left(t\right)\right)\right)$
are constructed in \cite{LSV05}, and termed \emph{Ramanujan Complexes}.
It is natural to ask whether these complexes are also optimal expanders
in the spectral and combinatorial senses.

\bibliographystyle{amsalpha}
\bibliography{cheeger}

\end{document}